%% file: GassGlau_coercivePIDE_2021.tex
\RequirePackage{scrlfile}
\ReplacePackage{scrpage2}{scrlayer-scrpage}
\documentclass[
BCOR2pt, 
captions=nooneline, 
bibliography=totoc, 
numbers=noenddot, 
parskip=half, 
headings=normal, 
abstracton 
]{scrartcl} 

\usepackage{setspace}

\usepackage{todonotes}
\usepackage{hyperref}

\usepackage{wallpaper}
\usepackage{mathdots}
\usepackage{color} 
\usepackage[T1]{fontenc} 
\usepackage[USenglish]{babel} 
\usepackage{graphicx} 
\usepackage{remreset} 
\usepackage[nouppercase]{scrpage2} 
\usepackage{amsmath} 
\usepackage{amssymb} 

\usepackage{natbib} 
\usepackage[
thmmarks, 
amsmath 
]{ntheorem} 
\usepackage{bm} 
\usepackage{bbm} 
\usepackage{enumitem} 
\usepackage[hang]{subfigure} 
\usepackage{wrapfig} 
\usepackage{tabularx} 
\usepackage{dcolumn} 
\usepackage{booktabs} 
\usepackage{listings} 
\usepackage{psfrag} 
\usepackage[olditem,oldenum]{paralist}
\usepackage{graphicx}  
\usepackage{epstopdf}
\makeatletter
\newcommand\mytoday{\number\year-\ifcase\month\or 01\or 02\or 03\or 04\or 05\or 06\or 07\or 08\or 09\or 10\or 11\or 12\fi-\ifcase\day\or 01\or 02\or 03\or 04\or 05\or 06\or 07\or 08\or 09\or 10\or 11\or 12\or 13\or 14\or 15\or 16\or 17\or 18\or 19\or 20\or 21\or 22\or 23\or 24\or 25\or 26\or 27\or 28\or 29\or 30\or 31\fi} 
\makeatother
\setkomafont{sectioning}{\normalcolor\bfseries} 
\clearscrheadfoot 
\automark[section]{subsection} 
\setkomafont{captionlabel}{\bfseries} 
\newcolumntype{d}[2]{D{.}{.}{#1.#2}} 
\newcommand*{\abstractnoindent}{} 
\let\abstractnoindent\abstract
\renewcommand*{\abstract}{\let\quotation\quote\let\endquotation\endquote
\abstractnoindent}
\makeatletter
\renewcommand{\p@enumii}[1]{\theenumi(#1)}
\makeatother


\theoremstyle{break} 
\theoremheaderfont{\bfseries}
\theorembodyfont{\itshape}
\theoremseparator{}
\newtheorem{definition}{Definition}[section] 
\newtheorem{lemma}[definition]{Lemma}
\newtheorem{theorem}[definition]{Theorem}
\newtheorem{corollary}[definition]{Corollary}
\newtheorem{proposition}[definition]{Proposition}

\newtheorem{conditions}{Conditions}

\newtheorem{remark}[definition]{Remark}
\newtheorem{example}[definition]{Example}

\theoremstyle{nonumberbreak} 
\theoremsymbol{$\Box$}
\theorembodyfont{\upshape}
\newtheorem{proof}{Proof}

\usepackage{booktabs}
\usepackage{multirow}

\newcommand{\tild}{~}
\newcommand{\ee}[1]{\operatorname{e}^{#1}}
\newcommand{\OA}{\mathcal{A}}
\newcommand{\OF}{\mathcal{F}}
\newcommand{\OG}{\mathcal{G}}
\newcommand{\OM}{\mathcal{M}}

\newcommand*{\IR}{\mathbb{R}}
\newcommand*{\rr}{\mathbb{R}}
\newcommand*{\rrd}{\mathbb{R}^d}

\newcommand*{\IN}{\mathbb{N}}

\renewcommand*\d{\mathop{}\!\mathrm{d}}

\newcommand{\Dt}{\Delta t}
\newcommand{\orderA}{2\varrho}
\newcommand{\orderAhalf}{\varrho}

\newcommand{\norm}[1]{{\left\Vert #1 \right\Vert}}
\addto\extrasUSenglish{}
\addto\extrasUSenglish{}

\numberwithin{equation}{section}

\begin{document}

\include{titlepage}

\section{Introduction}
Solving partial (integro) differential equations (PI)DEs) is---besides Monte Carlo and Fourier techniques---one of the fundamental approaches to compute financial quantities in asset models based on jump processes. These quantities include option prices, sensitivities, risk measures and optimal investment strategies. While Monte Carlo type simulations typically are two slow for real-time evaluations, Fourier methods are known to be very efficient. These type of methods are, however, not directly applicable to  exotic options. PIDE methods have the potential to 
\begin{enumerate}
\item be flexible in the model choice,
\item be flexible in the option type,
\item allow for a thorough error control.
\end{enumerate}
An emerging field combines deep neural network approximations and P(I)DE techniques to compute option prices in high-dimensional settings, i.e.\ for a large number of dependent stochastic factors, see for instance \cite{HanJentzenE2018}, \cite{HanJentzenE2018_2}. Merging both techniques promises to merry the lucid mathematical insight in the error behaviour of PDE methods with the spectacular approximation power of machine learning, see for instance the recent analytic achievements of \cite{OpschoorPetersenSchwab2019}, and \cite{GononSchwab2020}. The breadth of this development justifies the further development of the error analysis of P(I)DE techniques, which can serve as a building block in the analysis of combined techniques.


Two essential ingredients of such an analysis are \emph{stability estimates} and the analysis of \emph{convergence rates}. A numerical scheme is said to be stable if the normed solution is bounded by the normed right-hand side and the initial condition as inputs. For linear equations this directly implies a robustness assertion, namely small perturbations of the inputs lead to small perturbations of the discrete solution. To ensure reliability it is therefore necessary that the scheme satisfies stability. Moreover, stability estimates typically are a key step in the derivation of the asymptotic rate of convergence of the scheme. The latter provides an assertion on the efficiency of the scheme and can be used to test the implementation, also in the frequent situation where no exact solution is available.

In this article, we advance the classical error analysis of PIDEs for option pricing. We consider the discretization of PIDEs with a Galerkin scheme in space and a theta Euler scheme in time. The loss distribution of individual assets typically displays fat tails. Moreover, large losses often are induced by general market shocks, as we experience during the current pandemic, and impact a large number of business entities simultaneously. Therefore we focus on L\'evy type models, which allow to capture both fat tails of log-returns and the dependence of large losses, i.e.\ tail dependence. Both features are absolutely essential in order to reproduce the risk of losses in a realistic manner. L\'evy type models prove to reproduce financial data highly successfully, particularly if the model displays one additional feature, namely non-stationarity of returns, see for instance \cite{EberleinKluge2006b} and\cite{EberleinOezkan2005}, \cite{CrepeyGrbacNgorSkovmand2015}, \cite{EberleinMadan2009} for valuation of structured products and  \cite{KokholmNicolato2010} for a model default probabilities. The class generalizing L\'evy processes to non-stationarity of increments is called time-inhomogeneous L\'evy processes, or processes with independent increments and absolutely continuous characteristics (PIIAC), additive processes and Sato processes in the literature. 



We therefore keep our analysis is kept in such a generality that it allows to jointly treat
\begin{enumerate}
\item[--] multivariate asset models (the dimension $d$ is acts as a parameter), 
\item[--] time-inhomogeneous L\'evy models,
\item[--] European and for instance barrier option types. 
\end{enumerate}
In order to achieve this level of generality, we exploit the Hilbert space structure of the Galerkin approach. 

In this article we derive stability and convergence results for finite element and more general Galerkin methods for parabolic evolution equations that in particular arise for option pricing in \emph{time-inhomogeneous L\'evy models}. 
Time-inhomogeneity of the modelling process translates to time dependence of the operator governing the evolution equation for option pricing. We therefore are dealing with parabolic evolution equations with \textit{time-dependent operators}.

Inspired by applications of the fractional Brownian motion in finance, \cite{Reichmann2011} examines evolution problems with time-dependent coefficients that  exhibit a technically challenging degeneracy at the initial time point. Here, we argue that the assumption of non-degenerate coefficients is widely applicable for pricing equations in finance. First, empirical evidence for such a degeneracy seems not to be available. Second, the case of piecewise constant coefficients that is frequently used in finance falls in this scope. Third, further examples and construction principles leading to time-inhomogeneous L\'evy processes with non-degenerate infinitesimal generators are provided in Example 7.6 by \cite{EberleinGlau2014} and in Section 4.3 by \cite{Glau2016}.
Finally, the class of Sato processes used in \cite{EberleinMadan2009} that falls out of this scope, has originally been introduced to better capture the distributional behaviour of the asset data for \textit{large maturities}. Therefore the modification of the model has been suggested in Example 7.7  by \cite{EberleinGlau2014}. The coefficients of the inifitesimal generator of the modified process are non-degenerate. 

The assumption of non-degeneracy allows us to use the techniques provided by \cite{PetersdorffSchwab2003}, who proved the respective assertions under for time-homogeneous operators.  
In order to contribute to a reliable application of the scheme, we particularly derive the arising time-step size condition in explicit form, keeping track of all constants involved. 
This analysis shows that, in contrast to the time-homogeneous case, the stability estimates and by consequence the time-step size conditions depend on the continuity and coercivity constant.

On the technical side, the main difference from our proofs compared to the proofs by \cite{PetersdorffSchwab2003} is the following. Their analysis strongly relies on the energy norm. For a time-dependent operator, however, the energy norm is replaced by a time-dependent family of norms with each member being an energy norm related to the operator at a fixed time point. Since dealing with a family of norms is rather cumbersome at some places, we define an appropriate auxiliary norm. 
To underline the high level of generality and in particular the separation of assumptions on the solution from assumptions on the approximation property of the Galerkin spaces, we base our analysis on general solution spaces and express the approximation property in a functional way. 

As our main results we establish stability and convergence rates for linear evolution equations governed by generators of \textit{time-inhomogeneous} L\'evy processes. In this article assume \textit{coercivity} of the resulting bilinear form. Typically, in financial applications we require the more general case of a non-coercive bilinear form that satisfy a G{\aa}rding inequality. 
We treat this general case in the follow-up paper \cite{GassGlau2020b}. The proof is considerably more involved and builds on the results provided in this article.  

The main contribution of the article is twofold.
Firstly, we extend the stability and convergence analysis provided in \cite{PetersdorffSchwab2003} to the time-in\-homo\-ge\-neous case. When trying to adapt \citeauthor{PetersdorffSchwab2003}'s proof, a technical complication arises from the energy norm $\norm{v_h}_{a}$ becoming time-dependent. This compels us to deviate from their concept of proof. A notable feature of our new proof is that our main results do no longer require the inverse property, one of the essential assumptions in \cite{PetersdorffSchwab2003}. In addition, we provide the final convergence results in terms of the projection error, which applies generally to choices of function spaces for the exact and the approximate solution. 
Secondly, we discuss all assumptions in regards to applicability in finance. Particularly, we present examples of function spaces for the exact and the approximate solution, which arise naturally for problems in finance.



The remainder of the article is organized as follows. 
In Section \ref{sec-pb}, we present the variational formulation of the evolution equation and the fully discrete solution scheme that we investigate. Moreover, we formulate the assumptions on the scheme and discuss them in the light of the application to option pricing in time-inhomogeneous L\'evy models. In order to make both the theoretical results as well as the concepts accessible to the community of financial mathematicians, we carefully introduce the mathematical objects. We also discuss all assumptions in the light of financial applications. In the subsequent Section \ref{sec-stab} we derive the stability estimate for the fully discrete scheme. The final Section \ref{sec:ConvCoercive} provides the convergence analysis. For the readers convenience we present all proofs that are omitted in the main sections in the appendix.

	\section{Problem Formulation}\label{sec-pb}
We first introduce the weak formulation of the following problem: Find solutions $u:[0,T]\times \IR^d \rightarrow \IR$ of the evolution equation
	\begin{equation}
	\label{eq:evol}
	\begin{split}
		\partial_t u + \mathcal{A}_t u =&\ f,\qquad\text{for almost all $t\in(0,T)$}\\
		u(0) =&\ g,
	\end{split}
	\end{equation}
with $\mathcal{A} = (\mathcal{A}_t)_{t\in[0,T]}$ a time-inhomogeneous \emph{Kolmogorov operator}, a \emph{source term} or \emph{right hand side} $f:[0,T]\times\IR^d\rightarrow \IR$ and an \emph{initial condition} $g:\IR^d \rightarrow \IR$.	

Convergence results typically hinge on the degree of regularity of the solution. We incorporate this in our notation by indexing the solution spaces $V^\varrho$ with a regularity parameter $\varrho>0$. 	
	Following the classical way to define solution spaces of parabolic evolution equations and fixing an index $\varrho>0$, we introduce a \emph{Gelfand triplet} $(V^\varrho,H,(V^\varrho)^\ast)$, which consists of a pair of separable Hilbert spaces $V^\varrho$ and $H$ and the dual space $(V^\varrho)^\ast$ of $V^\varrho$ such that there exists a continuous embedding from $V^\varrho$ into $H$.
We denote by $L^2\big(0,T; H\big)$ the space of weakly measurable functions $u:[0,T]\to H$ with $\int_0^T\|u(t)\|_H^2 \d{t} < \infty$ and by $\partial_t u$ the derivative of $u$  with respect to time in the distributional sense.
The Sobolev space 
\begin{equation*}
W^1( 0,T; V^\varrho,H) := \Big\{ u\in L^2\big(0,T;V^\varrho\big) \,\Big| \,\partial_t u\in L^2\big(0,T; (V^\varrho)^\ast\big) \Big\},
\end{equation*}
will serve as solution space for equation \eqref{eq:evol}. For a more detailed introduction to the space $W^1\big(0,T; V^\varrho, H\big)$, which relies on the Bochner integral, we refer to Section 24.2 in \cite{Wloka-english}.  More information on Gelfand triplets can be found for instance in Section 17.1 in \cite{Wloka-english}.

Let $a: [0,T]\times V^\varrho \times V^\varrho \to \rr$ be a family $a=(a_t)_{t\in[0,T]}$ of bilinear forms that are measurable with respect to $t$ and let $\OA=(\OA_t)_{t\in[0,T]}$ with $\OA_t: V^\varrho \to (V^\varrho)^\ast$ be a family of operators. We say \textit{$a$ is associated with $\OA$}, if for every $t\in[0,T]$,
\begin{equation}\label{rel_OAa}
\OA_t(u)(v) = a_t(u,v)\qquad\text{for all } u,v\in  V^\varrho.
\end{equation}

We consider PIDEs with time-dependent operator $\mathcal{A}$  associated with a family of bilinear forms $a_t(\cdot,\cdot): V^\varrho \times V^\varrho \rightarrow \mathbb{R}$ for each $t\in[0,T]$ that is \textit{continuous and coercive uniformly in time}, a notion that we precise in the following two definitions.
\begin{definition}[Continuity and coercivity uniformly in time]
\label{def:acontinuous}
A bilinear form $a_\cdot(\cdot,\cdot): [0,T] \times V^\varrho \times V^\varrho \rightarrow \mathbb{R}$ is called 
\begin{enumerate}
\item[(i)] \emph{continuous} uniformly in time with respect to $V^\varrho$, if there exists $\alpha\in\mathbb{R}^+$ such that
	\begin{equation}
	\label{eq:defatcont}
		|a_t(u,v)| \leq \alpha \norm{u}_{V^\varrho}\norm{v}_{V^\varrho}
	\end{equation}
holds for all $u,v \in V^\varrho$ and for all $t\in [0,T]$. 
\item[(ii)] \emph{coercive} uniformly in time with respect to $V^\varrho$, if there exists $\beta\in\mathbb{R}^+$ independent of $t$ such that
	\begin{equation}
	\label{eq:defatcoerc}
		a_t(u,u) \geq \beta \norm{u}^2_{V^\varrho}
	\end{equation}
holds for all $u \in V^\varrho$ and for all $t\in [0,T]$. 
\end{enumerate}
We call such a constant $\alpha$ a \emph{continuity constant} and such a constant $\beta$ a \emph{coercivity constant} of the bilinear form $a$.
\end{definition}


Next we provide the weak formulation of evolution equation~\eqref{eq:evol}. We denote by $\langle \cdot, \cdot\rangle_{H}$ the scalar product  in $H$ and by $\langle \cdot | \cdot\rangle_{(V^\varrho)^\ast\!\times V^\varrho}$ the dual pairing mapping from $(V^\varrho)^\ast\!\times V^\varrho$ to $\rr$.

\begin{definition}
Let   $f\in L^2\big(0,T; (V^\varrho)^\ast\big)$ and $g\in H$. Then 
$u\in W^1( 0,T; V^\varrho,H)$ is a \emph{weak solution} of the evolution equation \eqref{eq:evol}, if for almost every $t\in(0,T)$, 
\begin{equation}\label{def-para}
 \langle \partial_t u(t), v\rangle_{H} + a_{T-t}( u(t), v) =\, \langle f(t) | v\rangle_{(V^\varrho)^\ast\!\times V^\varrho}\quad \text{for all }v\in V^\varrho
\end{equation}
and $u(t)$ converges to $g$ for $t\downarrow0$ in the norm of $H$.
\end{definition}

If the bilinear form $a$ is continuous and coercive with respect to $V^\varrho,H$ both uniformly in time then a classical existence and uniqueness result for parabolic equations yields the existence and uniqueness of a weak solution $u\in W^1( 0,T; V^\varrho,H)$ of equation \eqref{eq:evol}, see for instance Theorem~23.A in \cite{Zeidler}.

	\subsection{The Discrete Scheme}
	We introduce the discretisation along with notation that we use in the sequel.

\begin{definition}[Semi-discrete weak solution]
\label{def:WeakSolutionSemiDiscrete}
Let $V$, $H$ be separable Hilbert spaces and the dual $V^\ast$ of $V$ forming a Gelfand triplet,
	\begin{equation*}
		V \hookrightarrow H \cong H^\ast \hookrightarrow V^\ast.
	\end{equation*}
Let $V_h\subset V$ be finite dimensional and $f\in L^2(0,T; V^\ast)$. We call $u_h\in W^1(0,T; V_h, H)$ a \emph{semi-discrete weak solution} to problem~\eqref{eq:evol}, if for almost every $t\in(0,T)$
	\begin{equation}
	\begin{split}
		(\partial_t u_h(t),v_h)_H + a_t(u_h(t),v_h) =&\ \langle f(t), v_h\rangle_{V^\ast\times V},\qquad u_h(0) = g_h
	\end{split}
	\end{equation}
holds for all $v_h\in V_h$, where the time derivative is understood in the weak sense, $a$ is the bilinear form associated with operator $\mathcal{A}$ and $g_h\in H$.
\end{definition}


To obtain the fully discrete problem formulation we  discretize the time horizon $[0,T]$ equidistantly.
%
%
Choose $M\in\IN$ and define $\Dt = T/M$ and $t^m = \Dt\,m$ for all $m\in\{0,\dots,M\}$. We call $(T,M,\Dt)$ an \emph{equidistant time discretization}, the set $\{t^0,t^1,\dots,t^M\}$ the associated \emph{equidistant time grid} and we call $\Dt$ the \emph{time stepping size}. Throughout the following, the number of time steps will always be denoted by $M$ and $\Dt$ will always be defined as above.

\begin{definition}[Fully discrete weak solution]
\label{def:WeakSolutionFullyDiscrete}
Let $V$, $H$ separable Hilbert spaces and the dual $V^\ast$ of $V$ be given that form a Gelfand triplet,
	\begin{equation*}
		V \hookrightarrow H \cong H^\ast \hookrightarrow V^\ast
	\end{equation*}
and let $V_h\subset V$ a finite dimensional subspace of $V$. Let $f\in L^2(0,T; V^\ast)$. Further choose $M\in\IN$ and let $\{t^0,\dots,t^M\}$ an equidistant time grid with time stepping size $\Dt$. Finally choose $\theta\in[0,1]$. Then we call $(u^m_h)_{m\in\{0,\dots,M\}}$, $u_h^m\in V_h$, the \emph{fully discrete weak solution} to problem~\eqref{eq:evol}, if
	\begin{equation}
	\label{eq:ThetaSchemeFullydiscretizedPIDE}
	\begin{split}
		\left(\frac{u_h^{m+1}-u_h^m}{\Dt}, v_h\right)_H + a^{m+\theta}(u^{m+\theta}_h(t),v_h) =&\ \langle f^{m+\theta}, v_h\rangle_{V^\ast\times V},\qquad u^0_h = g_h
	\end{split}
	\end{equation}
holds for all $v_h\in V_h$ and for all $m\in\{0,\dots,M-1\}$ for some $g_h\in H$, 
and where 
	\begin{align}
		u_h^{m+\theta} =&\ \theta u_h^{m+1} + (1-\theta)u_h^m,\label{eq:defuhmtheta}\\
		f^{m+\theta} =&\ \theta f^{m+1} + (1-\theta)f^m,\label{eq:deffmtheta}
	\end{align}
with $f^m = f(t^m)$. With $a$ being the bilinear form associated with operator $\mathcal{A}$ we have set
	\begin{equation}
	\label{eq:defamtheta}
		a^{m+\theta}(\cdot,\cdot) = a_{\theta t^{m+1} + (1-\theta)t^m}(\cdot,\cdot).
	\end{equation}
An iterative relation between $u_h^m$ and $u_h^{m+1}$ for all $m\in\{0,\dots,M-1\}$ as arising from \eqref{eq:ThetaSchemeFullydiscretizedPIDE} is also called \emph{$\theta$ scheme}.
\end{definition}

In general, the two parameters $\theta\in[0,1]$ and $M\in\IN$  of the Theta Scheme~\eqref{eq:ThetaSchemeFullydiscretizedPIDE} can not be chosen independently from each other. The variable $M$ serves as a measure of the fineness of the discretization $(T,M,\Dt)$ in time. The value of $\theta$ controls the degree of implicitness of the scheme \eqref{eq:ThetaSchemeFullydiscretizedPIDE}. With $\theta=1$ the element $u_h^{m+1}$, appears twice in the scheme \eqref{eq:ThetaSchemeFullydiscretizedPIDE}, which is then called \emph{fully implicit}. With $\theta=0$ the element $u_h^{m+1}$ appears only once and thus the scheme is called \emph{fully explicit}. So called \emph{semi-explicit} schemes are those with $\theta\in(0,1)$ with the Crank-Nicolson scheme as the most prominent example $(\theta=\frac{1}{2})$.
As we will see later, in case that $\theta\leq \frac{1}{2}$, convergence and stability lemmas and theorems only grant their claims if $\Dt$ is small enough. Conditions of that sort are always called \emph{time stepping conditions}.
For the accuracy of an approximate solution $(u_h^m)_{m\in\{0,\dots,M\}}$ to problem~\eqref{eq:evol}, also the approximation quality of $g_h$, the approximate of the initial value $g$ plays a vital role.

%
%

	\subsection{Preliminaries and Assumptions}
\label{sec:AssumptionSection}

The conditions under which we derive stability and convergence are given below. In the way they are stated, they generalize the set of assumptions required by~\cite{PetersdorffSchwab2003} for their stability and convergence analysis.

For the error analysis and the derivation of convergence results we assume that the space $V$ that the solution space of the weak solution $u\in W^1(0,T;V,H)$ is built on provides a certain smoothness, denoted by a positive real value $s\in\IR^+$. More precisely, the space $V$ will always be a Sobolev space with index $s\in \IR^+$, see Definition in Equation~\eqref{eq:Hetastilde} as an example. From here on, we therefore add the superscript $s$ to $V$, thus wrting $V^\varrho$ instead of $V$, and its finite dimensional subspaces, by writing analogously $V^\varrho_h$ instead of $V_h$, to represent the smoothness of the respective space.
We discretize the space $V^\varrho$ to receive a finite dimensional subspace $V^\varrho_h\subset V^\varrho$ using for example piecewise polynomials of degree $p\geq 0$.


\begin{conditions}\label{cond}
\emph{For some indexes $0<\varrho<t$, a Gelfand triplet $(V^\varrho,H,(V^\varrho)^\ast)$ and an additional space $V^t\supset V^\varrho$, the bilinear form $a_\cdot(\cdot,\cdot): [0,T] \times V^{\varrho} \times V^{\varrho} \rightarrow \mathbb{R}$ and a finite dimensional subspace $V^\varrho_h\subset V^\varrho$, we introduce the following set of conditions:
\vspace{-0.5ex}
\begin{enumerate}[leftmargin=3em, label=(A\arabic{*}),widest=(A4)]\label{A1-A3}
\item
(Continuity and coercivity)  is continuous and coercive, both uniformly in time with respect to $V^{\varrho}$.
\item\label{ass:GeneralApproximationProperty}
(Approximation property of the Galerkin space)
There exists a family of bounded linear projectors
	$P_h:V^\varrho\rightarrow V^\varrho_h$, a constant $C_\Upsilon>0$ 
	and a function $\Upsilon$ with $\Upsilon(h,\varrho,t,u)\rightarrow 0$ for $h\rightarrow 0$ 
	such that for all $u\in V^t$,
	\begin{equation}
	\label{eq:GeneralApproximationProperty}
		\norm{u-P_h u}_{V^\varrho} \leq C_\Upsilon\,\Upsilon(h,\varrho, t, u) .
	\end{equation}
\item\label{ass:quasioptinitial}
(Quasi-optimality of the initial condition)
 There is a constant $C_I>0$ independent of $h>0$ such that
	\begin{equation}
		\norm{g - g_h}_H\leq C_I\inf\limits_{v_h \in V^\varrho_h} \norm{g - v_h}_H.
	\end{equation}
\end{enumerate}
}
 \end{conditions}
Condition (A1) is equivalent to the continuity and ellipticity of the bilinear form $a$ with respect to $V^\varrho$. We formulate Condition (A2)in terms of a function $\Upsilon(h,\varrho, t, u)$ since often this function often is of the form $\Upsilon(h,\varrho, t, u)= h^{t-\varrho}\|u\|_{V^t}$ and more general expressions also appear.
Conditions (A2),(A3) are basic approximation conditions on the Galerkin spaces. They are not only satisfied for $V^\varrho_h$ being the linear space spanned by the hat functions with mesh fineness $h$, but also for instance for wavelet approximation spaces. We will discuss all conditions in detail in Section~\ref{sec-discuss-assumptions}.

In order to perform the stability and convergence estimates, an appropriate norm for the dual space of $V^\varrho$ is needed. In the case of time-homogeneity the natural candidate for such a norm is $\norm{f}_{\ast}  :=\ \sup\limits_{\underset{v_h\neq 0}{v_h \in V^\varrho_h}} \frac{(f,v_h)}{\norm{v_h}_{a}}$ for  all $f\in {V^\varrho}^\ast$. This norm is chosen in \cite{PetersdorffSchwab2003} and turns out to be very appealing since repeatedly terms
In our case, however, the energy norm $\norm{v_h}_{a}$ needs to be replaced by $\norm{v_h}_{a_t}$, the time-dependent family of norms that turns out to be not suitable. This forces us to deviate from the concept used in \cite{PetersdorffSchwab2003}, and we need to define another norm for the dual space of the discrete solution space. To do so, the following observation proves useful.

\begin{remark}[Energy norm]
A bilinear form $a$ that is continuous and coercive both uniformly in time in the sense of Definition \ref{def:acontinuous} induces a norm $\norm{\cdot}_{a_t} = \sqrt{a_t(\cdot,\cdot)}$ on $V^\varrho$ for each $t\in[0,T]$ that is equivalent to the norm of $V^\varrho$, since 
	\begin{equation*}
		\sqrt{\beta}\norm{u}_{V^\varrho} \leq \norm{u}_{a_t} \leq \sqrt{\alpha}\norm{u}_{V^\varrho},
	\end{equation*}
for all $u\in V^\varrho$ wherein $\alpha$ and $\beta$ are the time independent constants from Definition~\ref{def:acontinuous}. The norm $\norm{\cdot}_{a_t}$ is called \emph{enery norm} of $a_t(\cdot,\cdot)$. In contrast to the time-homogeneous case the energy norm is time-dependent leading to a whole family of norms. 
\end{remark}

Attached to the spaces $H$, $V^\varrho$ and $V^\varrho_h$ we introduce the norms
	\begin{equation}
	\begin{alignedat}{2}
		\norm{u} :=&\ \norm{u}_{H},&&\qquad \text{for $u\in H$},\\
		\norm{f}_{{V^\varrho_h}^\ast}  :=&\ \sup\limits_{\underset{v_h\neq 0}{v_h \in V^\varrho_h}} \frac{(f,v_h)}{\norm{v_h}_{V^\varrho}},&&\qquad \text{for $f\in {V^\varrho}^\ast$}.\label{eq:Vsstarnorm}
	\end{alignedat}
	\end{equation}
Moreover, we will also need the constant $\Lambda$, defined by
	\begin{equation}
	\label{eq:defLambda}
		\Lambda	:= \sup\limits_{\underset{v_h\neq 0}{v_h \in V^\varrho_h}} \frac{\norm{v_h}_H^2}{\norm{v_h}_{{V^\varrho_h}^*}^2}.
	\end{equation}

\begin{remark}[On $\Lambda$]
\label{rem:LambdaFinite}
Given $h>0$ and the respective finite dimensional space $V^\varrho_h\subset V^\varrho$, the constant $\Lambda$ defined in \eqref{eq:defLambda} is finite due to the fact that all norms involved are norms restricted to finite dimensional spaces and thus all norms are equivalent. From this, $\Lambda$ being finite follows immediately. Notice, however, that $\Lambda$ depends on $h$ and thus on the dimension of the spaces involved,
	\begin{equation*}
		\Lambda = \Lambda(h).
	\end{equation*}
Yet, $\Lambda$ is not necessarily bounded in $h$ and its limit for $h\rightarrow 0$ does not need to be finite.
\end{remark}

As explained above, we are forced to deviate from the proof of \cite{PetersdorffSchwab2003} due to the time dependence of the operator. Interestingly, our new proof does not rely on one of the essential assumptions in \cite{PetersdorffSchwab2003}, namely we do not require the inverse property stated below.

\begin{enumerate}[leftmargin=3em, label=(A\arabic{*}),widest=(A4)]\label{A4}
\setcounter{enumi}{3}
\item\label{ass:InverseProperty}
(Inverse property)
There is a constant $C_\text{IP}>0$ independent of $h>0$ such that for all $u_h\in V^\varrho_h$
	\begin{equation}
		\norm{u_h}_{V^\varrho} \leq C_\text{IP}\,h^{-\varrho}\norm{u_h}_H.
	\end{equation}
	\end{enumerate}



We are now in a position to state the main results.

\section{Stability Estimates}\label{sec-stab}
We derive a stability result regarding a solution to $\theta$-Scheme~\eqref{eq:ThetaSchemeFullydiscretizedPIDE} under the assumption of continuity and coercivity of the associated time-dependent bilinear form.

\begin{proposition}[Stability estimate for $\theta$-Scheme]
\label{prop:StabThetaCoercive}
Let $a_\cdot(\cdot,\cdot)$ be a time-dependent bilinear form that is both continuous and coercive uniformly in time with respect to $V^\varrho$ and $H$ with continuity constant $\alpha$ and coercivity constant $\beta$. 
Let $\theta\in [0,1]$ and let $(u_h^m)_{m\in\{0,\dots,M\}}$ be a solution of the associated $\theta$-Scheme~\eqref{eq:ThetaSchemeFullydiscretizedPIDE} on an equidistant time grid $(T, M, \Dt)$. For $\theta\in \left[\frac{1}{2},1\right]$ let
	\begin{align*}
		0 < C_1& < 2,\qquad
		 C_2 \geq \frac{1}{\beta(2-C_1)}.		
	\end{align*}
For $\theta \in \big[0,\frac{1}{2}\big)$ assume the time stepping size $\Dt$ to satisfy the time stepping condition
	\begin{equation}
	\label{eq:stabcoercivedt}
		0<\Dt < \frac{2\beta}{(1-2\theta)\Lambda\alpha^2}.
	\end{equation}
Define the constant
	\begin{equation}
	\label{eq:stabcoercivemu}	
		\mu := \left(1-2\theta\right)\Lambda\Dt > 0
	\end{equation}
and let 
	\begin{align}
		C_1& \in \left(0,2-\frac{\mu\alpha^2}{\beta}\right),\qquad
		C_2 \geq \max\left\{\mu, \frac{(1+\mu\alpha)^2}{(2-C_1)\beta - \mu\alpha^2} + \mu\right\}.\label{eq:stabcoerciveC2}
	\end{align}
Then the following stability estimate holds,
	\begin{equation*}
		\norm{u_h^M}_H^2 + \Dt\, C_1  \sum_{m=0}^{M-1} \norm{u_h^{m+\theta}}_{a^{m+\theta}}^2 \leq \norm{u_h^0}_H^2 + \Dt\, C_2 \sum_{m=0}^{M-1}\norm{f^{m+\theta}}_{{V^\varrho_h}^\ast}^2.
	\end{equation*}
\end{proposition}

The following remark argues that the intervals for the constants $C_1, C_2$ introduced  in the proposition are indeed well-defined.

\begin{remark}[On the constants of Proposition~\ref{prop:StabThetaCoercive}]
For $\theta\in[0,\frac{1}{2})$ the constant $\mu$ is well defined and indeed larger than zero and the set of possible values for $C_1$, $C_2$ is non-empty. With $\Dt$ chosen according to \eqref{eq:stabcoercivedt} we have
	\begin{equation*}
		\frac{\mu\alpha^2}{\beta} = \frac{(1-2\theta)\Lambda\Dt\alpha^2}{\beta} <\frac{(1-2\theta)\Lambda\frac{2\beta}{(1-2\theta)\Lambda\alpha^2}\alpha^2}{\beta} = 2
	\end{equation*}
which admits a non-empty interval from which $C_1$ may be chosen. Since $\mu$ is finite and bounded, we have $C_2<\infty$ if $(2-C_1)\beta > \mu\alpha^2$.
Observe that the latter is true by considering the range from which $C_1$ may be chosen.
\end{remark}

\begin{remark}[On the time stepping condition and the inverse property]
\label{rem:LambdaInvProp}
Consider the time stepping condition~\eqref{eq:stabcoercivedt} for $\theta \in \big[0,\frac{1}{2}\big)$ in Proposition~\ref{prop:StabThetaCoercive}. Under the inverse property Condition~\ref{ass:InverseProperty} we have for all $w_h\in V^\varrho_h$,
	\begin{equation}
	\begin{split}
		\norm{w_h}_{{V^\varrho_h}^\ast} =&\ \sup\limits_{v_h\in V^\varrho_h}\frac{(w_h,v_h)}{\norm{v_h}_{V^\varrho}} 
		\geq \frac{1}{C_\text{IP}}h^s\sup\limits_{v_h\in V^\varrho_h}\frac{(w_h,v_h)}{\norm{v_h}_H} \geq \frac{1}{C_\text{IP}}h^{\varrho}\norm{w_h}_H
	\end{split}
	\end{equation}
and hence
	\begin{equation}
	\begin{split}
		\Lambda =&\ \sup\limits_{v_h\in V^\varrho_h} \frac{\norm{v_h}_H^2}{\norm{v_h}_{{V^\varrho_h}^\ast}^2} \leq C^2_\text{IP}\,h^{-2\varrho},
	\end{split}
	\end{equation}
with $\Lambda=\Lambda(h)$ defined in \eqref{eq:defLambda}. Consequently, under Condition~\ref{ass:InverseProperty}, for $\theta \in \big[0,\frac{1}{2}\big)$ the time stepping condition on $\Dt$ as required by \eqref{eq:stabcoercivedt} in Proposition~\ref{prop:StabThetaCoercive} is satisfied if
	\begin{equation}
		0 < \Dt < \frac{2\beta}{(1-2\theta)C^2_\text{IP}\,\alpha^2}\,h^{\orderA}= C_\theta\,h^{\orderA}
	\end{equation}
with $C_\theta = 2\beta/[(1-2\theta)C^2_\text{IP}\,\alpha^2]$.
\end{remark}
Finally, we consider a stability result with respect to the norm $\norm{\cdot}_{V^\varrho}$. 

\begin{corollary}[Stability of the $\theta$-scheme]
\label{cor:StabThetaCoerciveVnorm}
Under the assumptions of Proposition~\ref{prop:StabThetaCoercive} the stability estimate
	\begin{equation}
	\label{eq:StabThetaCoerciveVnorm}
		\norm{u_h^M}_H^2 + \Dt\, C_1\beta  \sum_{m=0}^{M-1} \norm{u_h^{m+\theta}}_{V^\varrho}^2 \leq \norm{u_h^0}_H^2 + \Dt\, C_2 \sum_{m=0}^{M-1}\norm{f^{m+\theta}}_{{V_h^s}^\ast}^2.
	\end{equation}
\end{corollary}

\begin{proof}
The claim is a direct consequence of Proposition~\ref{prop:StabThetaCoercive} and the uniform coercivity of the bilinear form with coercivity constant $\beta$.
\end{proof}

 Corollary~\ref{cor:StabThetaCoerciveVnorm} shows
that the solution of the discrete scheme is bounded by its initial data in a discrete $L^2(0,T,V^\varrho)$ or $L^2(0,T,{V^\varrho_h}^\ast)$ norm, respectively.

\section{Convergence Analysis}
\label{sec:ConvCoercive}

In this subsection we derive a converge result for the $\theta$-Scheme~\eqref{eq:ThetaSchemeFullydiscretizedPIDE}. 
For that matter we consider the residuals between each member of $(u^m)_{m\in\{0,\dots,M\}}$, the weak solution of \eqref{eq:evol} evaluated at time points $t^m$, $m=0,\dots,M$, and the respective members of the sequence $(u^m_h)_{m\in\{0,\dots,M\}}$, the solution of $\theta$-Scheme~\eqref{eq:ThetaSchemeFullydiscretizedPIDE}.

In order to ultimately prove convergence, we will show that (parts of) these residuals satisfy an auxiliary $\theta$-scheme. The crucial observation is that this scheme is of a similar structure as $\theta$-Scheme~\eqref{eq:ThetaSchemeFullydiscretizedPIDE}. This enables us to apply Proposition \ref{prop:StabThetaCoercive} to these very residuals will yield an upper bound for the sum of their norms from which convergence can be deduced. With this structure of the proof we closely follow \cite{PetersdorffSchwab2003}. We provide all proofs that follow along the same lines as in \cite{PetersdorffSchwab2003} in detail in the appendix. Within this section, we only present very short proofs and the proof of the final convergence result, which requires deviating from the concept in \cite{PetersdorffSchwab2003}.

We define for all $m\in\{0,\dots,M\}$ the difference $e_h^m$ between the weak solution evaluated at time point $t^m$ and its finite dimensional approximation affiliated with time point $t^m$ as
	\begin{equation}
	\label{eq:defehm}
	\begin{split}
		e_h^m =&\ u^m - u_h^m\\
			=&\ (u^m - P_h u^m) + (P_h u^m - u_h^m)\\
			=&\ \eta^m + \xi_h^m,
	\end{split}
	\end{equation}
with
	\begin{align}
		\eta^m =&\ u^m - P_h u^m,\qquad \forall m\in\{0,\dots, M\}, \label{eq:defetam}\\
		\xi_h^m =&\ P_h u^m - u_h^m,\qquad \forall m\in\{0,\dots, M\}, \label{eq:defxim}
	\end{align}
with a projector $P_h$ adhering to Assumption~\ref{ass:GeneralApproximationProperty}. The quantity $e_h^m$ thus consists of two parts. The first part, $\eta^m$, carries the discretization error, the second part, $\xi_h^m$, denotes the inaccuracy of the approximate solution with respect to the projection of the weak solution into the finite dimensional subspace.

Our final result shows that the convergence rate depending on the the projection error to the Galerkin space, i.e.\ the function $\Upsilon$ of Assumption~\ref{ass:GeneralApproximationProperty}, and its behaviour when $h$ tends to zero. This behaviour of the projection error in turn originates from the smoothness that the weak solution $u$ admits. The more regularity it exhibits, the faster the achieved rate of convergence will be. 
In Section~\ref{sec-approxprop} we will discuss the behaviour of the projection error for examples of solution spaces and Galerkin spaces, which are appropriate for problems in finance.

To derive the convergence result we focus on the term $\xi_h^m$ in~\eqref{eq:defehm}, first. Being the part of the residual $e_h^m$ that denotes the deviation of the solution of the $\theta$ scheme from the projection of the weak solution, it is of central interest for the whole analysis.

\begin{lemma}[$\theta$-scheme for the $\xi_h^m$]
\label{lem:residual}
Let $u\in W^1(0,T;V^\varrho,H)$ be the weak solution to problem~\eqref{eq:evol} with continuous and coercive bilinear form $a$, i.e.\ we assume Condition~(A1). Furthermore, be $(u_h^m)_{m\in\{0,\dots,M\}}$ the solution to $\theta$-Scheme~\eqref{eq:ThetaSchemeFullydiscretizedPIDE}, and let $\xi_h^m$, $m\in\{1,\dots,M\}$, be defined by \eqref{eq:defxim}. If additionally $u\in C^1([0,T];H)$ and the bilinear form is continuous in $t$ then we have
	\begin{equation}
	\begin{split}
		\left(\frac{\xi_h^{m+1} - \xi_h^m}{\Dt}, v_h\right) + a^{m+\theta}(\theta \xi_h^{m+1} + (1-\theta) \xi_h^m,v_h) =&\ (r^m,v_h),\quad
		\xi_h^0 = P_h g - u_h^0
	\end{split}
	\end{equation}
for all $m=1,\dots,M-1$ and for all $v_h \in V^\varrho_h$, where the weak residuals $r^m: V^\varrho_h\rightarrow\mathbb{R}$ have the form	
		\begin{equation}
		\label{Theta:xihm}
			r^m = r_1^m + r_2^m + r_3^m
		\end{equation}
for all $m\in\{0,\dots,M-1\}$, with
	\begin{align*}
		(r_1^m, v_h) =&\ \left(\frac{u^{m+1}-u^m}{\Dt} - \dot{u}^{m+\theta}, v_h\right),\\
		(r_2^m, v_h) =&\ \left(\frac{P_h u^{m+1} - P_h u^m}{\Dt} - \frac{u^{m+1} - u^m}{\Dt}, v_h\right),\\
		(r_3^m, v_h) =&\ a^{m+\theta}\left(P_h u^{m+\theta} - u^{m+\theta}, v_h\right).
	\end{align*}
\end{lemma}

The proof of the lemma is provided in Appendix \ref{sec-theta_residual}.



For the solution $(\xi_h^m)_{m\in\{0,\dots,M\}}$ of Scheme~\eqref{Theta:xihm}, the following stability estimate holds.

\begin{corollary}[Stability estimate for $\xi_h^m$]
\label{cor:stabxicoerc}
Let $(\xi_h^m)_{m\in\{0,\dots,M\}}$ be the solution of $\theta$-Scheme~\eqref{Theta:xihm} with $\theta\in[0,1]$ and let the assumptions of Proposition~\ref{prop:StabThetaCoercive} be satisfied.
Then there exist positive constants $C_1$, $C_2$ such that the following stability estimate holds,
	\begin{equation}
	\label{eq:stabxicoerc}
		\norm{\xi_h^M}_H^2 + \Dt\, C_1  \sum_{m=0}^{M-1} \norm{\xi_h^{m+\theta}}_{a^{m+\theta}}^2 \leq \norm{\xi_h^0}_H^2 + \Dt\, C_2 \sum_{m=0}^{M-1}\norm{r^m}_{{V^\varrho_h}^\ast}^2.
	\end{equation}
\end{corollary}

\begin{proof}
By assumption, the bilinear form $a_t(\cdot,\cdot)$ is continuous and coercive uniformly in time. The $\xi_h^m$ thus take the role of the $u^m_h$ in the $\theta$ scheme~\eqref{eq:ThetaSchemeFullydiscretizedPIDE} and the $r^m$ take the role of the $f^{m+\theta}$ therein. Consequently, we can directly apply Proposition~\ref{prop:StabThetaCoercive}. The constants $C_1$, $C_2$ of the corollary are thus identical to the two constants of the lemma.
\end{proof}

To derive convergence of the approximate solution we will show convergence of the right hand side in \eqref{eq:stabxicoerc}. In that respect, Corollary~\ref{cor:stabxicoerc} is the key ingredient to our convergence results for bilinear forms that are both continuous as well as coercive uniformly in time. In preparation of these results we shall now derive upper bounds for the individual residual parts $r^m_1, r^m_2$ and $r^m_3$.


\begin{lemma}[Upper bounds for normed residuals]
\label{lem:UpperBoundsforresiduals}
Let (A1)--(A3) of Conditions~\ref{cond} be satisfied and let $(r_i^m,\cdot)_H$ with $r_i^m :V^\varrho_h\rightarrow\IR$, $i\in\{1,2,3\}$, be the weak residuals derived by the lemma. We require additional smoothness of the weak solution $u$ by assuming further that
	\begin{enumerate}[label=\roman*)]
		\item $u\in W^1(0,T;V^t, H)$ for some $t\geq \orderAhalf$,\label{enum:uinWVt}
		\item $u\in C^2([0,T], H)$.\label{enum:uC2} 
	\end{enumerate}
In case $\theta=\frac{1}{2}$ assume optionally
	\begin{enumerate}[label=\roman*)]
		\setcounter{enumi}{2}
		\item $u\in C^3([0,T], H)$. \label{enum:UpperBoundsOptional}
	\end{enumerate}
Then there exist positive constants $C_{r_1}$, $C_{r_2}$ and $C_{r_3}$ such that
	\begin{align}
		\norm{r_1^m}_{{V^\varrho_h}^\ast} \leq&\  C_{r_1}\begin{cases} \sqrt{\Dt} \left(\int_{t^m}^{t^{m+1}} \norm{\ddot{u}(s)}^2_{{V^\varrho_h}^\ast}  \d{s}\right)^\frac{1}{2},&\quad \theta\in [0,1] \\
		\Dt^\frac{3}{2} \left(\int_{t^m}^{t^{m+1}} \norm{\dddot{u}(s)}^2_{{V^\varrho_h}^\ast}  \d{s} \right)^\frac{1}{2},&\quad \theta = \frac{1}{2}\text{ and given \ref{enum:UpperBoundsOptional} holds} \\
		\end{cases} \label{eq:UpperBoundsforresidualsr1}\\
		\norm{r_2^m}_{{V^\varrho_h}^\ast} \leq&\ C_{r_2}\,\frac{1}{\sqrt{\Dt}} \left(\int_{t^m}^{t^{m+1}} \Upsilon^2\left(h, \orderAhalf, t, \dot{u}(\tau)\right)\d{\tau}\right)^\frac{1}{2}, \\
		\norm{r_3^m}_{{V^\varrho_h}^\ast} \leq&\ C_{r_3}\,\Upsilon(h, \orderAhalf, t, u^{m+\theta})
	\end{align}
for all $m=0,\dots,M-1$.
\end{lemma}

The proof is provided in Appendix~\ref{sec-proofUpperBoundsforresiduals}.

We are now able to state the core theorem, granting convergence of the $\theta$-Scheme~\eqref{eq:ThetaSchemeFullydiscretizedPIDE}, where the involved bilinear form is continuous and coercive uniformly in time.

\begin{theorem}[Convergence of the coercive $\theta$ scheme]
\label{thrm:ConvergenceCoercive}
We assume Condition \ref{cond} holds and that  $u\in W^1(0,T;V^t,H)$ is the weak solution to problem~\eqref{eq:evol}.  Furthermore, we assume
	\begin{enumerate}[label=\roman*)]
		\item $u$ to be smooth enough in the sense that $u\in C^2([0,T],H)$,
		\item and for $\theta\in[0, 1/2)$ let the time stepping condition~\eqref{eq:stabcoercivedt} be satisfied.
	\end{enumerate}
In case $\theta=\frac{1}{2}$ assume optionally
	\begin{enumerate}[label=\roman*)]
		\setcounter{enumi}{2}
		\item $u\in C^3([0,T],H)$. \label{enum:CoercCondConvuC3}
	\end{enumerate}
Be $(u_h^m)_{m\in\{0,\dots,M\}}$ the solution to the associated $\theta$-Scheme~\eqref{eq:ThetaSchemeFullydiscretizedPIDE} with $\theta \in [0,1]$.

Then there exists a constant $\overline{C}>0$ such that
\begin{equation}
	\begin{split}
		\norm{u^M - u_h^M}^2 + \Dt\sum_{m=0}^{M-1}&\norm{u^{m+\theta} - u_h^{m+\theta}}_{a^{m+\theta}}^2\\
		\leq&\ \overline{C}\,\max\limits_{0\leq \tau\leq T}\Upsilon^2(h,\orderAhalf, t, u(\tau))\\
		+&\ \overline{C}\,\int_0^T \Upsilon^2(h,\orderAhalf,t,\dot{u}(\tau))\d{\tau}\\
		+&\ \overline{C}\,{\begin{cases}
							(\Dt)^2 \int_0^T \norm{\ddot{u}(s)}^2_{{V^\varrho_h}^\ast} \d{s}, & \forall \theta \in [0,1]\\
							(\Dt)^4 \int_0^T \norm{\dddot{u}(s)}^2_{{V^\varrho_h}^\ast}\d{s}, & \theta =\frac{1}{2}\text{ and with \ref{enum:CoercCondConvuC3}.}
						\end{cases}}
	\end{split}
	\end{equation}
\end{theorem}
We notice that the norms $\norm{\ddot{u}}_{{V^\varrho_h}^\ast}$ and $\norm{\ddot{u}}_{{V^\varrho_h}^\ast}$ on the right hand side can be replaced by $\norm{\ddot{u}}_{{V^\varrho}^\ast}$ and $\norm{\ddot{u}}_{{V^\varrho}^\ast}$ since $\le \norm{u}_{{V^\varrho}^\ast}$
since $\norm{u}_{{V^\varrho_h}^\ast} \le \norm{u}_{{V^\varrho}^\ast}$.

The proof of Theorem \ref{thrm:ConvergenceCoercive} is provided in Appendix \ref{sec:proofthrm:ConvergenceCoercive}.

\begin{corollary}[Convergence of the $\theta$ scheme]
\label{cor:ConvThetaCoerciveVnorm}
Under the assumptions of Theorem \ref{thrm:ConvergenceCoercive} there exists $\overline{C}>0$ such that
\begin{equation}
	\begin{split}
		\norm{u^M - u_h^M}^2 +\ \Dt\sum_{m=0}^{M-1}&\norm{u^{m+\theta} - u_h^{m+\theta}}_{V}^2\\
		\leq&\ \overline{C}\,\max\limits_{0\leq \tau\leq T}\Upsilon^2(h, \orderAhalf, t u(\tau))\\
		+&\ \overline{C}\,{\begin{cases}
							(\Dt)^2 \int_0^T \norm{\ddot{u}(\tau)}^2_{{V^\varrho_h}^\ast} \d{\tau}, & \forall \theta \in [0,1]\\
							(\Dt)^4 \int_0^T \norm{\dddot{u}(\tau)}^2_{{V^\varrho_h}^\ast}\d{\tau}, & \theta =\frac{1}{2}\text{ and with \ref{enum:CoercCondConvuC3}}
						\end{cases}}\\
		 +&\ \overline{C}\,\int_0^T \Upsilon^2(h, \orderAhalf, t, \dot{u}(\tau))\d{\tau}
	\end{split}
	\end{equation}
\end{corollary}

\begin{proof}
The result is an immediate consequence from Theorem \ref{thrm:ConvergenceCoercive} and the fact that $a_t(\cdot,\cdot)$ is coercive uniformly in time with coercivity constant $\beta$.
\end{proof}

In Theorem~\ref{thrm:ConvergenceCoercive} and Corollary~\ref{cor:ConvThetaCoerciveVnorm}, respectively, we have derived abstract convergence results. The particular convergence now follows immediately, when the form of $\Upsilon$ of Assumption~\ref{ass:GeneralApproximationProperty}, the general approximation property, is specified. The following corollary combines the result of 

\begin{corollary}[Convergence with $\Upsilon$ of \cite{PetersdorffSchwab2003}]
\label{cor:ConvergenceCoerciveSchwab}
Under the assumptions of Theorem~\ref{thrm:ConvergenceCoercive} and in the setting of \cite{PetersdorffSchwab2003} outlined in Example~\ref{ex:SchwabApproximationProperty}  there exists a constant $\overline{C}>0$ such that
\begin{align}
		\norm{u^M - u_h^M}^2 + \Dt&\sum_{m=0}^{M-1}\norm{u^{m+\theta} - u_h^{m+\theta}}_{a^{m+\theta}}^2\notag\\
		\leq&\ \overline{C}\,h^{2(p+1-\orderAhalf)}\,\max\limits_{0\leq \tau\leq T}\norm{u(\tau)}^2_{\mathcal{H}^{p+1}(\Omega)}\notag\\
		+&\ \overline{C}\,h^{2(p+1-\orderAhalf)}\,\int_0^T \norm{\dot{u}(\tau)}^2_{\mathcal{H}^{p+1}(\Omega)}\d{\tau}\\
		+&\ \overline{C}\,{\begin{cases}
							(\Dt)^2 \int_0^T \norm{\ddot{u}(s)}^2_{{V_h^{\orderAhalf}}^\ast} \d{s}, & \forall \theta \in [0,1]\\
							(\Dt)^4 \int_0^T \norm{\dddot{u}(s)}^2_{{V_h^{\orderAhalf}}^\ast}\d{s}, & \theta =\frac{1}{2}\text{ and if $u\in C^3([0,T],H)$.}
						\end{cases}}\notag
	\end{align}
\end{corollary}

\begin{proof}
The result is a direct consequence from Theorem~\ref{thrm:ConvergenceCoercive} with
	\begin{equation*}
		\Upsilon(h, s, t, u) = h^{t-s}\norm{u}_{\mathcal{H}^t(\Omega)},
	\end{equation*}
taking $t\leq p+1$ equal to its maximal admissible value with $p$ the polynomial degree that the basis functions of $V^{\orderAhalf}_h$ achieve piecewisely.
\end{proof}

The result of Corollary~\ref{cor:ConvergenceCoerciveSchwab} confirms the order of convergence derived in Theorem~5.4 of \cite{PetersdorffSchwab2003}. In contrast to that former result which our analysis is based on, we allow for time-dependent bilinear forms and thus generalize their result to the time-inhomogeneous case.




	\section{Discussion of the Main Assumptions}\label{sec-discuss-assumptions}
To set the assumptions introduced in the last section into the perspective of their applicability to option pricing equations in time-inhomogeneous L\'evy models let us first introduce some necessary notation.

Let a stochastic basis $(\Omega,\OF_T, (\OF_t)_{0\le t\le T}, P)$ be given and let
$L$ be an $\rrd$-valued \emph{time-inhomogeneous L\'evy process} with characteristics $(b_t,\sigma_t,F_t;h)_{t\ge0}$. Here, for every $s>0$, $\sigma_s$ is a symmetric, positive semi-definite $d\times d$-matrix, $b_s\in \rr^d$, and $F_s$ is a L\'evy measure, i.e. a positive Borel measure on $\rrd$ with $F_s(\{0\})=0$ and $\int_{\rr^d} (|x|^2 \wedge 1) F_s(\d{x})  < \infty$. Moreover, $h$ is a truncation function i.e. $h:\rr^d\to\rr$ such that $\int_{\{|x|>1\}} |h(x)| F_t(\d{x})<\infty$ with $h(x)=x$ in a neighbourhood of $0$. As usual we assume the maps $s\mapsto \sigma_s$, $s \mapsto b_s$ and $s\mapsto\int (|x|^2\wedge 1) F_s(\d{x})$ to be Borel-measurable with, for every $T>0$,
\begin{equation*}
\int_0^{T} \Big( |b_s| + \|\sigma_s\|_{\OM(d\times d)} +\int_{\rr^d} (|x|^2 \wedge 1) F_s(\d{x}) \Big) \d{s} < \infty,
\end{equation*}
 where $ \|\cdot\|_{\OM(d\times d)}$ is a norm on the vector space formed by the $d\times d$-matrices.

The process $L$ has independent increments and for fixed $t\ge0$ its \emph{infinitesimal generator} is given by
\begin{align}\label{def-A}
\begin{split}
\OG_t \varphi(x):= &  \frac{1}{2}\sum_{j,k=1}^d \sigma^{j,k}_t\frac{\partial^2 \varphi}{\partial x_j\partial x_k}(x) + \sum_{j=1}^d b^j_t\frac{\partial \varphi}{\partial x_j}(x)\\
& + \int_{\rr^d}\Big( \varphi(x+y)-\varphi(x)-\sum_{j=1}^d\frac{\partial \varphi}{\partial x_j}(x) \,  h_j(y)\Big)F_t(\d{y})
\end{split}
\end{align}
for every $\varphi\in C^\infty_0(\rrd)$, where $h_j$ denotes the $j$-th component of the truncation function\tild$h$.

\subsection{Solution spaces for different option types and models}	
We consider the stock price modelled by $S_t=S_0\ee{L_t}$.
Let
\begin{equation}\label{Kolm-op}
\OA_t\varphi(x):=-\OG_t\varphi(x) + r_t\varphi(x) + \kappa_t(x)\varphi(x)
\end{equation}
with deterministic, measurable and bounded \textit{interest rate} $r_{(\cdot)}:[0,T]\to\rr$ and measurable \textit{killing rate} $\kappa_{(\cdot)}(\cdot):[0,T]\times\rrd\to\rr$. \cite{Glau2016} establishes sufficient conditions on the L\'evy process and the killing rate yielding a Feynman-Kac representation of the weak solution $u$ in $W^1( 0,T; V^\varrho,H)$
of evolution equation \eqref{eq:evol} of the form
\begin{equation}\label{gl-stochdarst}
 \begin{split}
  u(T-t,x)
&=
E_{t,x}\Big(g(L_T)\ee{-\int_t^T  (r_h+ \kappa_h(L_{h})) \d{h}} \\
&
\qquad\qquad + \int_t^{T} \! f(T-\tau,L_\tau)\ee{-\int_t^\tau   (r_h+\kappa_h(L_{h})) \d{h}} \d{\tau}\Big).
 \end{split}
\end{equation}
For $\kappa\equiv0$ and $f\equiv0$, \eqref{gl-stochdarst} expresses the price of a plain vanilla option with payout function $g$. The additional freedom provided by $\kappa$ can for example be used to define a class of employee options.

Here the $V^\varrho$ is a weighted Sobolev-Slobodeckii space with an index $\varrho\in[1,2]$ and $H$ is a weighted $L^2$ space. More precisely, the \emph{exponentially weighted Sobolev-Slobodeckii space} $H^s_\eta(\rrd)$ with an index $s\ge0$ and a weight $\eta\in \rrd$ is the completion of $C_0^\infty(\rrd)$ with respect to the norm $\|\cdot\|_{H^s_\eta}$ given\tild by
\begin{equation*}
\|\varphi\|_{H^s_\eta}^2:= \int_{\rrd} \big(1+|\xi|\big)^{2s}\big|\OF(\varphi)(\xi - i \eta)\big|^2 \d{\xi}.
\end{equation*}
For $s=0$ the space $H^s_\eta(\rrd)$ coincides with the weighted space 
$L^2_\eta(\rr^d) := \{ u\in L^1_{loc}(\rr^d)\, | \, x\mapsto u(x)\ee{\langle \eta, x\rangle } \in L^2(\rr^d) \}$. 
A related class of function spaces that plays a fundamental role in this context are the spaces $\widetilde{H}^s_\eta(D)$ for open subsets $D\subset\rrd$ that are defined by
\begin{equation}
\label{eq:Hetastilde}
\widetilde{H}^s_\eta(D):=\big\{ v \in H^s_\eta(\rrd)\,\big|\, v|_{D^c}=0 \big\}.
\end{equation}
A Feynman-Kac type formula by \cite{PhdGlau} characterizes prices of barrier options in L\'evy models by weak solutions in these spaces. Second, typically the first step in the discretization of the evolution equation \eqref{eq:evol} is the truncation of the equation to a finite domain. This is formalized by restricting equation \eqref{def-para} for $V^\varrho:= H^\varrho_\eta(\rrd)$ to functions $v$ from $\widetilde{H}^\varrho_\eta(D)$ respectively from $\widetilde{H}^\varrho_0(D)$ for a bounded open domain in $D\subset\rrd$.

%

Examples of driving time-inhomogeneous L\'evy processes $L$ with the corresponding Sobolev-Slobodeckii spaces  as solution spaces can be found in \cite{EberleinGlau2014} and \cite{Glau2016}.
These include jump diffusions as well as pure jump processes such as of tempered stable, normal inverse Gaussian, student-$t$, Cauchy and Generalized Hyperbolic type.



Last, let us mention the more general class of anisotropic Sobolev-Slobodeckii spaces. For a large class of time-homogeneous multivariate processes \cite{HilberReichmannSchwabWinter2013} show existence  and uniqueness of the weak solution of the related evolution equation in anisotropic Sobolev-Slobodeckii spaces.

Notice that the operators of the type described in this subsection, when setting the killing rate to zero, overlap with the so-called fractional Laplace operator.

	\subsection{Coercivity}
	
	While the operators that arise in finance typically are not coercive and only satisfy a G{\aa}rding inequality, we restrict our analysis to operators that are coercive uniformly in time. Yet, a simple time transformation allows us to transform a linear PIDE whose operator only satisfies a G{\aa}rding inequality into a linear PIDE with coercive operator. To this end, consider the PIDE
	\begin{equation}
	\label{eq:PIDEgarding}
	\begin{split}
		\partial_t u + \mathcal{A}_t^\text{G{\aa}rding}u =&\ f,\\
		u(0) =&\ g,
	\end{split}
	\end{equation}
with weak solution $u\in W^1(0,T;V^\varrho,H)$, an operator $\mathcal{A}_t^\text{G{\aa}rding}$ that is assumed to be both continuous and of G{\aa}rding type uniformly in time with respect to the space $V^\varrho$ and that is associated with a bilinear form
	\begin{equation}
	\label{eq:bilineargarding}
		a_{(\cdot)}^\text{G{\aa}rding}: [0,T] \times V^\varrho \times V^\varrho \rightarrow \mathbb{R},\qquad (t,u,v)\mapsto a_t(u,v),
	\end{equation}
that fulfills a G{\aa}rding inequality uniformly in time, i.e.\
there exists constants $\beta,\lambda>0$ independently of $t$ such that
	\begin{equation}
	\label{eq:agarding}
		a_t^\text{G{\aa}rding}(u,u) \geq \beta \norm{u}^2_{V^\varrho} -\lambda \norm{u}_{H}
	\end{equation}
holds for all $u \in V^\varrho$ and for all $t\in [0,T]$.
Furthermore, let $a^\text{G{\aa}rding}$ be continuous with continuity constant denoted by $\alpha$.  Then, defining
	\begin{equation}
	\label{eq:uandulambda}
	\begin{split}
		u_\lambda(t,x) =&\ e^{-\lambda t}u(t,x),\qquad \forall (t,x)\in [0,T]\times\IR\\
		f_\lambda(t,x) =&\ e^{-\lambda t} f(t,x),\qquad \forall (t,x)\in [0,T]\times\IR
	\end{split}
	\end{equation}
and inserting \eqref{eq:uandulambda} into \eqref{eq:PIDEgarding}  and performing elementary manipulations yields
	\begin{equation}
	\begin{split}
	\label{eq:PIDEcoercified}
		\partial_t u_\lambda(t,x) + \mathcal{A}_{(\lambda,t)} u_\lambda(t,x) =&\ f_\lambda(t,x)\\
		u_\lambda(0) =&\ g,
	\end{split}
	\end{equation}
when we define the operator $\mathcal{A}_{(\lambda,\cdot)}$ by
	\begin{equation}
		\mathcal{A}_{(\lambda,\cdot)} = \left(\mathcal{A}_\cdot^\text{G{\aa}rding} + \lambda\right).
	\end{equation}
In contrast to $a_{(\cdot,\cdot)}^\text{G{\aa}rding}(\cdot,\cdot)$ of \eqref{eq:bilineargarding}, the associated bilinear form
	\begin{equation}
		a_{(\lambda,\cdot)} (\cdot,\cdot) : [0,T] \times V^\varrho \times V^\varrho \rightarrow \mathbb{R},\qquad (t,u,v)\mapsto {a_\lambda}_t(u,v),
	\end{equation}
is now coercive uniformly in time which we indicate by the subscript $\lambda$. The coercivity constant of $a_{(\lambda,\cdot)}$ is $\beta$, its continuity constant is given by $\alpha_\lambda = \alpha+ \lambda$. 


This elementary transformation shows that parabolic equations whose bilinear forms satisfy continuity and a G{\aa}rding inequality (both uniformly in time) can be transformed to parabolic problems whose bilinear form is continuous and coercive (both uniformly in time). This justifies the restriction to coercive problems. For practical applications, however, this approach still imposes a considerable restriction. The resulting error analysis only applies to the discrete scheme of the modified equation \eqref{eq:PIDEcoercified}, and not to the original$\theta$-Scheme \ref{Theta:xihm}. In order to implement the $\theta$-Scheme of the modified equation, the constant $\lambda$ must be available explicitly. Moreover, the multiplication of the sought-for solution $u$ by $\ee{\lambda t}$ has a negative influence on the condition number of the problem since different times $t$ are weighted differently. For these reasons the treatment of the more
general non-coercive case is of interest as well. Generalizing the proofs to a non-coercive bilinear form that satisfies the G{\aa}rding inequality, however, considerably complicates the matter. This case is treated in the separate article \cite{GassGlau2020b}, where the analysis rests on the results derived here.

%
%

\subsection{Approximation property}	
	\label{sec-approxprop}
	
\begin{example}[The setting of \cite{PetersdorffSchwab2003}]
\label{ex:SchwabApproximationProperty}
We present a first example of a specific instance for Assumption~(A2). In \cite{PetersdorffSchwab2003}, the authors consider the space
	\begin{equation*}
		\mathcal{H}^s(\Omega) = \begin{cases}
									V = \widetilde{W}^{\orderAhalf}(\Omega), & s=\orderAhalf\\
									V \cap H^s(\Omega), & s>\orderAhalf,
								\end{cases}
	\end{equation*}
for $\Omega\subset \IR^d$ a bounded domain with Lipschitz boundary $\Gamma = \partial\Omega$ and $\orderAhalf\in[0,1]$ the order of the possibly nonlocal operator $\mathcal{A}$ in problem~\eqref{eq:evol} with the space $\widetilde{W}^s(\Omega)$ defined as
	\begin{equation*}
		\widetilde{W}^s(\Omega) = \{u|_\Omega\,\big|\, u\in H^s(\IR^d),\, u|_{\IR^d\backslash\Omega} = 0\}.
	\end{equation*}
The discrete approximation $(u_h^m)_{m\in\{0,\dots,M\}}$ lies in $V_h\in\{V_h\}_{h>0}\subset V$, a finite dimensional subspace based on piecewise polynomials of degree $p\geq 0$, see Section~3.4.1 in \cite{PetersdorffSchwab2003} for details. Equation~(3.3) in \cite{PetersdorffSchwab2003} assumes that for all $u\in\mathcal{H}^t(\Omega)$ with $t\geq \orderAhalf$ there exists a $u_h\in V_h$ such that for $0\leq s \leq \orderAhalf$ and $\orderAhalf \leq t \leq p+1$
	\begin{equation}
	\label{eq:SchwabApproximationProperty}
		\norm{u-u_h}_{\tilde{H}^s(\Omega)} \leq c\,h^{t-s}\norm{u}_{\mathcal{H}^t(\Omega)}
	\end{equation}
for some $c>0$. The general function $\Upsilon$ of Assumption~(A2) is thus defined as
	\begin{equation}
	\label{eq:SchwabUpsilon}
		\Upsilon(h, s, t, u) = h^{t-s}\norm{u}_{\mathcal{H}^t(\Omega)}
	\end{equation}
for all $u\in\mathcal{H}^t(\Omega)$ and $C_\Upsilon = c$.
\end{example}



We present a further example for Sobolev spaces with integer index that originates from the results of~\cite{Beirao2014}. 

\begin{example}[Results from \cite{Beirao2014}]
\label{ex:SettingBeirao}
Let $\Omega\subset\IR$ be a bounded domain. Let further $s,t\in\IN_0$ with $0\leq s\leq t\leq p+1$ with $p\in\IN$. Consider the space of B-splines with degree $p$ spanned over a partition $\widetilde{\Delta}$, confer~\cite{Beirao2014} for details. Then there exists a projector
	\begin{equation}
		\Pi_{p,\widetilde{\Delta}}\,:\,H^{p+1}(\Omega) \rightarrow S_p(\widetilde{\Delta})
	\end{equation}
from the Sobolev space $H^{p+1}(\Omega)$ onto $S_p(\widetilde{\Delta})$, the space spanned by B-splines with degree $p$ such that with
	\begin{equation}
		V = H^s(\Omega)
	\end{equation}
there exists a constant $C(p)>0$ such that for all $u\in H^t(\Omega)$
	\begin{equation}
		\norm{u-\Pi_{p,\widetilde{\Delta}}u}_{H^s(\Omega)} \leq Ch^{t-s}\norm{u}_{H^t(\Omega)}
	\end{equation}
holds.
\end{example}

The results of Example~\ref{ex:SettingBeirao} also extend to non-integer Sobolev spaces. Consider for example Theorem~2.3.2 in~\cite{Roop2006} or similarly Theorem~7.2 in~\cite{ErvinRoop2007} for a verification of the approximation property in a fractional Sobolev space setting. Moreover, we refer to \cite{Takacs2015}, \cite{KarkulikMelenk2015} and \cite{Du2013} for further results and examples on the abstract approximation property of Condition~(A2).

Additionally, consider Definition~1.9 of the Ph.D. thesis of \cite{Schoetzau1999}, where a projector~$\Pi_l^r$ is defined. In Theorem~1.19 and Corollary~1.20, the author then derives approximation results for that projector. These results present themselves in the spirit of Condition~(A2) and hold for integer and non-integer Sobolev spaces, respectively.


\appendix
\section{Proof for Stability Estimates}

\begin{proof}[of Proposition~\ref{prop:StabThetaCoercive}]
The proof follows the structure of the proof of Proposition 4.1 by \cite{PetersdorffSchwab2003} replacing their norm $\norm{\cdot}_\ast$ by norm $\norm{\cdot}_{{V^\varrho_h}^\ast}$ as defined in equation \eqref{eq:Vsstarnorm}. At the core of the proof lies verifying that
	\begin{equation}
	\label{eq:stabcoerciveXm}
		X^m := \norm{u_h^m}_H^2 - \norm{u_h^{m+1}}_H^2 - \Dt\, C_1\norm{u_h^{m+\theta}}_{a^{m+\theta}}^2 + \Dt\, C_2 \norm{f^{m+\theta}}^2_{{V^\varrho_h}^\ast}\geq 0,
	\end{equation}
for all $m\in\{0,\dots,M-1\}$, since summing up the $X^m$, $m\in\{0,\dots,M-1\}$, then yields
	\begin{equation*}
		\sum_{m=0}^{M-1} X^m = \norm{u_h^0}_H^2 - \norm{u_h^{M}}_H^2 - \Dt\, C_1 \sum_{m=0}^{M-1}\norm{u_h^{m+\theta}}_{a^{m+\theta}}^2 + \Dt\, C_2 \sum_{m=0}^{M-1}\norm{f^{m+\theta}}^2_{{V^\varrho_h}^\ast} \geq 0,
	\end{equation*}
from which by simple rearrangement of terms if follows that
	\begin{equation*}
		\norm{u_h^M}_H^2 +  \Dt\, C_1 \sum_{m=0}^{M-1} \norm{u_h^{m+\theta}}_{a^{m+\theta}}^2 \leq \norm{u_h^0}_H^2 + \Dt\, C_2 \sum_{m=0}^{M-1}\norm{f^{m+\theta}}_{{V^\varrho_h}^\ast}^2,
	\end{equation*}
which shows the claim.

Fix $m\in\{0,\dots,M-1\}$ and define
	\begin{equation}
	\label{eq:stabcoercivebarw}
		\bar{w} = u_h^{m+1} - u_h^m.
	\end{equation}
With this definition of $\bar{w}$ the $\theta$ scheme~\eqref{eq:ThetaSchemeFullydiscretizedPIDE} yields
	\begin{equation}
	\label{eq:stabcoercive2}
	\begin{split}
		(\bar{w},u_h^{m+\theta}) =&\ \Dt \left(-a^{m+\theta}\left(u_h^{m+\theta}, u_h^{m+\theta}\right) + \left(f^{m+\theta},u_h^{m+\theta}\right)\right)\\
		=&\ \Dt \left(-\norm{u_h^{m+\theta}}_{a^{m+\theta}}^2 + \left(f^{m+\theta},u_h^{m+\theta}\right)\right),
	\end{split}
	\end{equation}
Inserting the definition of the norm $\norm{\cdot}_{{V^\varrho_h}^\ast}$ 
and the Cauchy-Schwarz inequality yields	
	\begin{equation}
	\label{eq:stabcoercive4}
		(\bar{w},u_h^{m+\theta}) \leq \Dt \left(-\norm{u_h^{m+\theta}}_{a^{m+\theta}}^2 + \norm{f^{m+\theta}}_{{V^\varrho_h}^\ast} \norm{u_h^{m+\theta}}_{V^s}\right).
	\end{equation}	
From the definition of $\bar{w}$ in \eqref{eq:stabcoercivebarw} we get
	$u_h^{m+\theta} = \left(u_h^m + u_h^{m+1}\right)/2 + \left(\theta - \frac{1}{2}\right)\bar{w}$, respectively
	\begin{equation}
	\label{eq:stabcoercive1}
		u_h^{m+1} + u_h^m = 2u_h^{m+\theta} - (2\theta -1)\bar{w}.
	\end{equation}
Inserting the definition of $\bar{w}$ and \eqref{eq:stabcoercive1} we see that
	\begin{equation*}
		\norm{u_h^{m+1}}_H^2 - \norm{u_h^{m}}_H^2 = (u_h^{m+1} - u_h^m, u_h^{m+1} + u_h^m) = \left(\bar{w}, 2u_h^{m+\theta} - (2\theta -1)\bar{w}\right),
	\end{equation*}
such that by changing signs we arrive at
	\begin{equation}
	\label{eq:stabcoercive5}
		\norm{u_h^m}_H^2 - \norm{u_h^{m+1}}_H^2 = -2(\bar{w}, u_h^{m+\theta}) + (2\theta -1)(\bar{w} ,\bar{w}).
	\end{equation}
Invoking the upper bound of \eqref{eq:stabcoercive4} in the first summand yields
	\begin{equation}
	\label{eq:stabcoercive6}
	\begin{split}	
		-2(\bar{w}, u_h^{m+\theta})\ +&\ (2\theta -1)(\bar{w} ,\bar{w})\\
		\geq&\ 2\Dt \left(\norm{u_h^{m+\theta}}_{a^{m+\theta}}^2 - \norm{f^{m+\theta}}_{{V^\varrho_h}^\ast} \norm{u_h^{m+\theta}}_{V^\varrho}\right) + (2\theta - 1)\norm{\bar{w}}_H^2.
	\end{split}
	\end{equation}
Combining \eqref{eq:stabcoercive5} and \eqref{eq:stabcoercive6} we obtain
	\begin{equation}
	\label{eq:stabcoercive7}
	\begin{split}
		\norm{u_h^m}_H^2\ -&\ \norm{u_h^{m+1}}_H^2\\
		\geq&\ 2\Dt \left(\norm{u_h^{m+\theta}}_{a^{m+\theta}}^2 - \norm{f^{m+\theta}}_{{V^\varrho_h}^\ast} \norm{u_h^{m+\theta}}_{V^\varrho}\right) + (2\theta - 1)\norm{\bar{w}}_H^2.
	\end{split}
	\end{equation}
Now we have collected all prerequisites for analyzing $X^m$ of \eqref{eq:stabcoerciveXm}, and we deduce
	\begin{equation}
	\label{eq:stabcoercive8}
	\begin{split}
		X^m =&\ \norm{u_h^m}_H^2 - \norm{u_h^{m+1}}_H^2 - \Dt\, C_1\norm{u_h^{m+\theta}}_{a^{m+\theta}}^2 + \Dt\, C_2 \norm{f^{m+\theta}}^2_{{V^\varrho_h}^\ast}\\
		\geq&\ 2\Dt\left( \norm{u_h^{m+\theta}}_{a^{m+\theta}}^2 - \norm{f^{m+\theta}}_{{V^\varrho_h}^\ast} \norm{u_h^{m+\theta}}_{V^\varrho}\right) + (2\theta - 1)\norm{\bar{w}}_H^2\\
		&\ \quad\quad\quad -\Dt\, C_1\norm{u_h^{m+\theta}}_{a^{m+\theta}}^2 + \Dt\, C_2 \norm{f^{m+\theta}}^2_{{V^\varrho_h}^\ast}.
	\end{split}
	\end{equation}
Collecting terms gives
	\begin{equation}
	\label{eq:stabcoercive9}
	\begin{split}
	X^m
		 \geq&\ 
		(2\theta - 1)\norm{\bar{w}}_H^2\\
		&\ + \Dt\left[(2-C_1)\norm{u_h^{m+\theta}}_{a^{m+\theta}}^2 + C_2\norm{f^{m+\theta}}^2_{{V^\varrho_h}^\ast}  - 2 \norm{f^{m+\theta}}_{{V^\varrho_h}^\ast} \norm{u_h^{m+\theta}}_{V^\varrho}\right].
	\end{split}
	\end{equation}
By assumption, the bilinear form $a_t(\cdot,\cdot)$ is uniformly coercive with coercivity constant $\beta$, so
	\begin{equation}
	\label{eq:stabcoercive10}
		\norm{u_h^{m+\theta}}_{a^{m+\theta}}^2 = a^{m+\theta}\left(u_h^{m+\theta}, u_h^{m+\theta}\right) \geq \beta\norm{u_h^{m+\theta}}^2_{V^\varrho}.
	\end{equation}
Employing the assumption $C_1<2$ and inserting \eqref{eq:stabcoercive10} into \eqref{eq:stabcoercive9} gives
	\begin{equation}
	\label{eq:stabcoercive11}
	\begin{split}
X^m \geq&\ (2\theta - 1)\norm{\bar{w}}_H^2\\
		&\ + \Dt\left[(2-C_1)\beta\norm{u_h^{m+\theta}}^2_{V^\varrho} + C_2\norm{f^{m+\theta}}^2_{{V^\varrho_h}^\ast}  - 2 \norm{f^{m+\theta}}_{{V^\varrho_h}^\ast} \norm{u_h^{m+\theta}}_{V^\varrho}\right].
	\end{split}
	\end{equation}

Next, we distinguish two cases for $\theta\in[0,1]$.

First, assume $\theta\in {[}\frac{1}{2},1{]}$.
In this case, proceeding from \eqref{eq:stabcoercive11} the second binomial formula yields
	\begin{equation}
	\label{eq:stabcoercive12}
	\begin{split}
		X^m \geq&\ (2\theta - 1)\norm{\bar{w}}_H^2 + \Dt\Bigg[\left(\sqrt{(2-C_1)\beta}\norm{u_h^{m+\theta}}_{V^\varrho} - \sqrt{C_2} \norm{f^{m+\theta}}_{{V^\varrho_h}^\ast} \right)^2 \\
		&\ \quad + 2\left(\sqrt{(2-C_1)\beta C_2} -1\right)\norm{f^{m+\theta}}_{{V^\varrho_h}^\ast} \norm{u_h^{m+\theta}}_{V^\varrho}\Bigg].
	\end{split}
	\end{equation}
Now,
	\begin{equation}
	\label{eq:stabcoercive13}
		(2-C_1)\beta C_2 \geq 1 \Leftrightarrow C_2\geq \frac{1}{\beta(2-C_1)},
	\end{equation}
which is true by assumption. By the choice of $\theta$, $(2\theta - 1)\norm{\bar{w}}_H^2 \geq 0$, and by \eqref{eq:stabcoercive13} all other summands in \eqref{eq:stabcoercive12} are nonnegative as well, so $X^m \geq 0$, 
which proves the claim of the Proposition for $\theta \in \left[\frac{1}{2},1\right]$.

Second, assume $\theta\in {[}0,\frac{1}{2}{)}$. 
Here, $(2\theta - 1)<0$, which prohibits arguing like above.
By $\theta$-Scheme~\eqref{eq:ThetaSchemeFullydiscretizedPIDE}, we have
	\begin{equation}
	\label{eq:stabcoercive16}
		(\bar{w},v_h) = \Dt\left(-a^{m+\theta}\left(u_h^{m+\theta}, v_h\right) + \left(f^{m+\theta}, v_h\right)\right),
	\end{equation}	
for all $v_h\in V^\varrho_h$. Consequently,
	\begin{equation}
	\begin{split}
		\norm{\bar{w}}_{{V^\varrho_h}^\ast}  =&\ \sup\limits_{v_h\in V_h} \frac{(\bar{w},v_h)}{\norm{v_h}_{V^\varrho}}
		=\ \sup\limits_{v_h\in V_h} \frac{\Dt\left(-a^{m+\theta}\left(u_h^{m+\theta}, v_h\right) + \left(f^{m+\theta}, v_h\right)\right)}{\norm{v_h}_{V^\varrho}},
	\end{split}
	\end{equation}
which gives 
	\begin{equation}
	\label{eq:stabcoercive19}
	\begin{split}
		\norm{\bar{w}}_{{V^\varrho_h}^\ast} \leq&\ \Dt\left(\sup\limits_{v_h\in V_h} \frac{-a^{m+\theta}\left(u_h^{m+\theta}, v_h\right)}{\norm{v_h}_{V^\varrho}} + \sup\limits_{v_h\in V_h}\frac{(f^{m+\theta}, v_h)}{\norm{v_h}_{V^\varrho}} \right)\\
		=&\ \Dt \left( \norm{a^{m+\theta}\left(u_h^{m+\theta},\cdot\right)}_{{V^\varrho_h}^\ast}  + \norm{f^{m+\theta}}_{{V^\varrho_h}^\ast} \right).
	\end{split}
	\end{equation}	
Clearly, by taking the uniform continuity of $a_t(\cdot,\cdot)$ with respect to $\norm{\cdot}_{V^\varrho}$ into account we deduce
	\begin{equation}
	\label{eq:stabcoercive18}
	\begin{split}
		\norm{a^{m+\theta}\left(u_h^{m+\theta}, \cdot\right)}_{{V^\varrho_h}^\ast}  =&\ \sup\limits_{v_h\in V_h} \frac{a^{m+\theta}\left(u_h^{m+\theta}, v_h\right)}{\norm{v_h}_{V^\varrho}}\\
		\leq&\ \norm{u_h^{m+\theta}}_{V^\varrho} \sup\limits_{v_h\in V_h} \frac{\alpha\norm{v_h}_{V^\varrho}}{\norm{v_h}_{V^\varrho}} = \alpha \norm{u_h^{m+\theta}}_{V^\varrho},
	\end{split}
	\end{equation}
which we insert into \eqref{eq:stabcoercive19} to get
	\begin{equation}
	\label{eq:stabcoercive20}
	\begin{split}		
		\norm{\bar{w}}_{{V^\varrho_h}^\ast} \leq&\ \Dt \left(\alpha\norm{u_h^{m+\theta}}_{V^\varrho} + \norm{f^{m+\theta}}_{{V^\varrho_h}^\ast} \right).
	\end{split}
	\end{equation}
By the definition of $\Lambda$ in \eqref{eq:defLambda} we have the \glqq inverse estimate\grqq
	\begin{equation}
	\label{eq:stabcoercive21}
		\norm{\bar{w}}_H \leq \sqrt{\Lambda}\norm{\bar{w}}_{{V^\varrho_h}^\ast}.
	\end{equation}
	Notice that therefore, we do not require the inverse property, Assumption \eqref{A4}.
Assembling our results by combining \eqref{eq:stabcoercive20} with \eqref{eq:stabcoercive21} gives
	\begin{equation}
	\label{eq:stabcoercive23}
		\norm{\bar{w}}_H \leq \sqrt{\Lambda}\norm{\bar{w}}_{{V^\varrho_h}^\ast} \leq \sqrt{\Lambda} \Dt \left( \alpha \norm{u_h^{m+\theta}}_{V^\varrho} + \norm{f^{m+\theta}}_{{V^\varrho_h}^\ast} \right).
	\end{equation}
Altogether we obtain
	\begin{align}
		X^m	
		\geq&\ (2\theta - 1)\left(\sqrt{\Lambda} \Dt \left( \alpha \norm{u_h^{m+\theta}}_{V^\varrho} + \norm{f^{m+\theta}}_{{V^\varrho_h}^\ast} \right)\right)^2\label{eq:stabcoercive24}\\
		& + \Dt\left[(2-C_1)\beta\norm{u_h^{m+\theta}}^2_{V^\varrho} + C_2\norm{f^{m+\theta}}^2_{{V^\varrho_h}^\ast}  - 2 \norm{f^{m+\theta}}_{{V^\varrho_h}^\ast} \norm{u_h^{m+\theta}}_{V^\varrho}\right].\notag
	\end{align}
Expanding the squared brackets and collecting terms we obtain 
	\begin{equation}
	\label{eq:stabcoercive26}
	\begin{split}
X^m \geq&\ \Dt\Bigg[ \left[(2-C_1)\beta - (1-2\theta)\Dt\Lambda\alpha^2\right]\norm{u_h^{m+\theta}}_{V^\varrho}^2\\
		&\ \quad - 2\left[1+ (1-2\theta)\Dt\alpha\Lambda\right]\norm{f^{m+\theta}}_{{V^\varrho_h}^\ast} \norm{u_h^{m+\theta}}_{V^\varrho}\\
		&\ \quad + \left[C_2 - (1-2\theta)\Dt \Lambda\right]\norm{f^{m+\theta}}_{{V^\varrho_h}^\ast} ^2 \Bigg].
	\end{split}
	\end{equation}		
Recall that by definition in \eqref{eq:stabcoercivemu}, $\mu=\left(1-2\theta\right)\Lambda\Dt$, 
which turns \eqref{eq:stabcoercive26} into
	\begin{equation}
	\label{eq:stabcoercive26.2}
	\begin{split}
X^m \geq&\ \Dt\Bigg[ \left[(2-C_1)\beta - \mu\alpha^2\right]\norm{u_h^{m+\theta}}_{V^\varrho}^2
		- 2\left[1+ \mu\alpha\right]\norm{f^{m+\theta}}_{{V^\varrho_h}^\ast} \norm{u_h^{m+\theta}}_{V^\varrho}
		\\
		&\ \quad 
		+ \left[C_2 - \mu\right]\norm{f^{m+\theta}}_{{V^\varrho_h}^\ast} ^2 \Bigg].
	\end{split}
	\end{equation}
Define constants
	\begin{align}
		\gamma =&\ 1 + \mu\alpha,\label{eq:stabcoercivegamma}\\
		\delta =&\ C_2 - \mu,\label{eq:stabcoercivedelta}\\
		\kappa =&\ (2-C_1)\beta - \mu\alpha^2.\label{eq:stabcoercivekappa}
	\end{align}
Trivially, $\gamma>0$. We also have $\delta\geq0$, since $C_2\geq \mu$ by the first condition for $C_2$ in \eqref{eq:stabcoerciveC2}. Furthermore, we have $\kappa > 0$, since $C_1 < 2-\frac{\mu\alpha^2}{\beta}$
by the upper bound for the open interval of possible values for $C_1$ according to \eqref{eq:stabcoerciveC2}. Inserting the definitions of the nonnegative $\delta$ and the positive $\gamma$ and $\kappa$ into \eqref{eq:stabcoercive26.2} and applying the second binomial formula gives
	\begin{equation}
	\label{eq:stabcoercive27}
	\begin{split}
X^m \geq \Dt \Bigg[\Bigg(\sqrt{\kappa} \norm{u_h^{m+\theta}}_{V^\varrho} -&\ \sqrt{\delta}\norm{f^{m+\theta}}_{{V^\varrho_h}^\ast} \Bigg)^2
		\quad + 2\left(\sqrt{\kappa\delta} - \gamma\right)\norm{f^{m+\theta}}_{{V^\varrho_h}^\ast} \norm{u_h^{m+\theta}}_{V^\varrho}\bigg].
	\end{split}
	\end{equation}
This lower bound for $X^m$ 
is thus nonnegative, if $\sqrt{\kappa\delta} \geq \gamma$,
which, by the nonnegativity of the constants involved, is equivalent to $\kappa\delta \geq \gamma^2$.
The latter holds, if
	\begin{equation*}
			\left((2-C_1)\beta - \mu\alpha^2\right)(C_2-\mu) \geq (1+\mu\alpha)^2.
	\end{equation*}
This is indeed is the case, since $C_2 \geq \frac{(1+\mu\alpha)^2}{(2-C_1)\beta - \mu\alpha^2} + \mu,$
by the second condition for $C_2$ in \eqref{eq:stabcoerciveC2}. Therefore, $X^m\geq 0$ which finishes the proof.
\end{proof}

\section{Convergence Analysis: Proofs}\label{sec-proofsConvCoercive}

\subsection{Derivation of the $\theta$-scheme for the residuals $\xi_h^m$}\label{sec-theta_residual}

\begin{proof}[of Lemma \ref{lem:residual}]
By admitting time dependence of the bilinear form, this lemma generalizes Lemma 5.1~in~\cite{PetersdorffSchwab2003}. The proof therein provides very reliable guidelines along which we now derive our result, as well.

Choose $v_h\in V^\varrho_h\subset V^\varrho$ arbitrary but fix and $m\in\{0,\dots,M-1\}$ and compute
	\begin{align}
		\left(\frac{\xi_h^{m+1} - \xi_h^m}{\Dt}, v_h\right)&\ + a^{m+\theta}\left(\theta \xi_h^{m+1} + (1-\theta) \xi_h^m,v_h\right)\notag \\
		=&\ \bigg(\frac{(P_h u^{m+1} - u^{m+1}_h) - (P_h u^m - u^m_h)}{\Dt}, v_h\bigg)\notag\\
		&\qquad\qquad + a^{m+\theta}\left(\theta \left(P_h u^{m+1} - u_h^{m+1}\right) + (1-\theta) \left(P_h u^m - u_h^m\right),v_h\right)\notag 
	\end{align}
Rearranging and invoking the relation provided by the $\theta$-Scheme \eqref{eq:ThetaSchemeFullydiscretizedPIDE} 
yields
	\begin{align}
		&\left(\frac{P_h u^{m+1} - P_h u^m}{\Dt} , v_h\right) + a^{m+\theta}(P_h u^{m+\theta},v_h)  - \left(\left( \frac{u^{m+1}_h - u^m_h}{\Dt}, v_h\right) + a^{m+\theta}(u_h^{m+\theta},v_h) \right)\notag\\
		&\qquad=\ \left(\frac{P_h u^{m+1} - P_h u^m}{\Dt} - \dot{u}^{m+\theta} , v_h\right) + a^{m+\theta}(P_h u^{m+\theta},v_h)
		+(\dot{u}^{m+\theta},v_h) - (f^{m+\theta},v_h)\label{eq:L51help02}.
	\end{align}
By Equation~\eqref{def-para} and the assumption that $u\in C^1([0,T],H)$ and that the bilinear form $a$ is continuous in $t$, the fundamental theorem of variational calculus implies that
	\begin{equation}
	\label{eq:L51help}
		(\dot{u}^{m+\theta},v) + a^{m+\theta}\left(u^{m+\theta},v\right) = (f^{m+\theta},v),\qquad\forall v\in V^\varrho.
	\end{equation}
Recalling that $v_h\in V^\varrho_h\subset V^\varrho$ we combine \eqref{eq:L51help}  and \eqref{eq:L51help02} to get
	\begin{align}
		&\left(\frac{P_h u^{m+1} - P_h u^m}{\Dt} - \dot{u}^{m+\theta} , v_h\right) + a^{m+\theta}(P_h u^{m+\theta},v_h) +(\dot{u}^{m+\theta},v_h) - (f^{m+\theta},v_h)\notag\\
		&\quad=\ \left(\frac{P_h u^{m+1} - P_h u^m}{\Dt} - \dot{u}^{m+\theta} , v_h\right) + a^{m+\theta}(P_h u^{m+\theta},v_h) - a^{m+\theta}(u^{m+\theta},v_h)\label{eq:L51help03}
	\end{align}
Adding an artificial zero to \eqref{eq:L51help03} we arrive at 	
	\begin{align*}
		&\left(\frac{P_h u^{m+1} - P_h u^m}{\Dt} - \dot{u}^{m+\theta} , v_h\right) + a^{m+\theta}(P_h u^{m+\theta},v_h) - a^{m+\theta}(u^{m+\theta},v_h)\\
		&\qquad=\ \left(\frac{P_h u^{m+1} - P_h u^m}{\Dt} - \frac{u^{m+1} - u^m}{\Dt},v_h\right) + \left(\frac{u^{m+1} - u^m}{\Dt} - \dot{u}^{m+\theta} , v_h\right)\\
		&\qquad\qquad + a^{m+\theta}(P_h u^{m+\theta} - u^{m+\theta},v_h),\\
		&\qquad=\ (r^m,v_h),
	\end{align*}
with
	\begin{equation*}
		(r^m,v_h) := (r_1^m + r_2^m + r_3^m,v_h),
	\end{equation*}
wherein
	\begin{align*}
		(r_1^m,\cdot) =&\ \left(\frac{u^{m+1} - u^m}{\Dt} - \dot{u}^{m+\theta} , \cdot\right),\\
		(r_2^m,\cdot) =&\ \left(\frac{P_h u^{m+1} - P_h u^m}{\Dt} - \frac{u^{m+1} - u^m}{\Dt}, \cdot\right),\\
		(r_3^m,\cdot) =&\ a^{m+\theta}(P_h u^{m+\theta} - u^{m+\theta},\cdot),
	\end{align*}
which validates the decomposition of $r^m$ claimed by the lemma.
\end{proof}

\subsection{Derivation of the Upper Bounds of the Residuals}
\label{sec-proofUpperBoundsforresiduals}.

\begin{proof}[of Lemma \ref{lem:UpperBoundsforresiduals}]
The proof follows along the same guidelines as the derivation of the upper bounds of the residuals in \cite{PetersdorffSchwab2003}.
Choose $v_h\in V^\varrho_h$ arbitrary but fix. We derive each upper bound individually.
First, we derive an upper bound for $\norm{r_1^m}_{{V^\varrho_h}^\ast}$.
Clearly,
	\begin{equation}
	\label{eq:r1estimatestart}
		\left|(r_1^m,v_h)\right| \leq \norm{\frac{u^{m+1}-u^m}{\Dt} - \dot{u}^{m+\theta}}_{{V^\varrho_h}^\ast} \norm{v_h}_{V^\varrho},
	\end{equation}		
by the definition of the norm $\norm{\cdot}_{{V^\varrho_h}^\ast}$. Recall that our time grid is equidistantly spaced so $t^{m+1} = t^m + \Dt$ for all $m\in\{0,\dots,M-1\}$. Under the assumption that $u\in C^2\left([0,T],H\right)$ we represent $u^{m+1}=u(t^{m+1})$ by the Taylor expansion of $u$ around $t^m$, evaluated at $t^{m+1}$. Thus, by applying 
the Taylor theorem for Banach-valued functions, we have
	\begin{equation}
		u^{m+1} = u^m + \dot{u}^m \Dt + \int_{t^m}^{t^{m+1}}(t^{m+1} - \tau)\ddot{u}(\tau) \d{\tau}.
	\end{equation}
Proceeding with elementary calculations we get
	\begin{align}
		\frac{u^{m+1} - u^m}{\Dt} -&\ \dot{u}^{m+\theta}
		=&\ \frac{1}{\Dt}\int_{t^m}^{t^{m+1}}(t^{m+1}-\tau)\ddot{u}(\tau) \d{\tau} - \left(\theta \dot{u}^{m+1} - \theta\dot{u}^m\right).\label{eq:r1mineq1}
	\end{align}
Since $u\in C^2([0,T], H)$, the absolute continuity of continuously differentiable functions, compare \cite[Chapter IV]{Elstrodt}, together with Proposition 1.2.3 in \cite{ArendtBattyHieberNeubrander2011} grant that $\ddot{u}$ is Bochner integrable and
	\begin{equation}
	\label{eq:r1mineq2}
		\theta \dot{u}^{m+1} - \theta\dot{u}^m = \theta\int_{t^m}^{t^{m+1}} \ddot{u}(\tau)\d{\tau}.
	\end{equation}
Inserting \eqref{eq:r1mineq2} into \eqref{eq:r1mineq1} yields
	\begin{equation}
	\label{eq:r1mineq2.1}
	\begin{split}
		\frac{u^{m+1} - u^m}{\Dt} - \dot{u}^{m+\theta} 
		=&\ -\frac{1}{\Dt}\int_{t^m}^{t^{m+1}}\left(\tau - (1-\theta)t^{m+1} - \theta t^m\right)\ddot{u}(\tau) \d{\tau}.
	\end{split}
	\end{equation}
Since $\ddot{u}$ is Bochner integrable, we can apply 
the results from \cite[Section 24.2]{Wloka-english} to get
	\begin{align}
		\norm{\frac{u^{m+1} - u^m}{\Dt} - \dot{u}^{m+\theta}}_{{V^\varrho_h}^\ast}
		 &= \norm{\frac{1}{\Dt}\int_{t^m}^{t^{m+1}}\left(\tau - (1-\theta)t^{m+1} - \theta t^m\right)\ddot{u}(\tau) \d{\tau}}_{{V^\varrho_h}^\ast}\notag\\
		& \leq \frac{1}{\Dt} \int_{t^m}^{t^{m+1}}\norm{\left(\tau - (1-\theta)t^{m+1} - \theta t^m\right) \ddot{u}(\tau)}_{{V^\varrho_h}^\ast}\d{\tau}\label{eq:r1mineq3}
	\end{align}
taking the norm into the integral. Considering the function
	\begin{equation*}
		g_\theta(\tau) = \tau - (1-\theta)t^{m+1} - \theta t^m,\quad g_\theta : \left[t^m,\ t^{m+1}\right] \rightarrow \mathbb{R}
	\end{equation*}
we find due to its strict monotonicity in $\tau$ that for $\tau\in [t^m, t^{m+1}]$
	\begin{equation}
	\label{eq:r1mineq3.1}
	\begin{split}
		\left|g_\theta(\tau)\right| \leq&\ \max\left\{\left|t^{m+1} - (1-\theta)t^{m+1} - \theta t^m\right|, \left|t^m - (1-\theta)t^{m+1} - \theta t^m\right|\right\}\\
		=&\ \max\{|\theta(t^{m+1} - t^m)|, |(1-\theta)(t^{m+1} - t^m)|\} \\
		=&\ \Dt\max\{\theta, (1-\theta)\} = \Dt\, C_\theta,
	\end{split}
	\end{equation}
with $C_\theta = \max\{\theta, (1-\theta)\}$. Using the estimate \eqref{eq:r1mineq3.1} we develop \eqref{eq:r1mineq3} into
	\begin{align*}
		\frac{1}{\Dt} \int_{t^m}^{t^{m+1}}\norm{\left(\tau - (1-\theta)t^{m+1} - \theta t^m\right) \ddot{u}(\tau)}_{{V^\varrho_h}^\ast}\d{\tau}
		\leq&\ C_\theta \int_{t^m}^{t^{m+1}} \norm{\ddot{u}(\tau)}_{{V^\varrho_h}^\ast}  \d{\tau}\notag \\
		\leq&\ C_\theta \sqrt{\Dt} \left(\int_{t^m}^{t^{m+1}} \norm{\ddot{u}(\tau)}^2_{{V^\varrho_h}^\ast}  \d{\tau}\right)^\frac{1}{2},
	\end{align*}
with the H\"older inequality being applied in the last step. 

\emph{Special case $\theta = 1/2$}:
For $\theta=1/2$ and under the assumption of additional smoothness of $u$ in the sense of Assumption \ref{enum:UpperBoundsOptional} being satisfied, further computations are possible. Then by integration by parts and elementary calculations
	\begin{equation}
	\label{eq:r1mineq4}
	\begin{split}
		&\ \frac{u^{m+1} - u^m}{\Dt} - \dot{u}^{m+\theta}\\
		=&\ -\frac{1}{\Dt}\int_{t^m}^{t^{m+1}}\left(\tau - \frac{1}{2}t^{m+1} - \frac{1}{2} t^m\right)\ddot{u}(\tau) \d{\tau}\\
		=&\ -\frac{1}{2\Dt} \left(-t^m t^{m+1}\left(\ddot{u}^{m+1} - \ddot{u}^m\right) - \int_{t^m}^{t^{m+1}}\left(\tau^2 - (t^{m+1} + t^m)\tau\right)\dddot{u}(\tau) \d{\tau}\right) \\
		=&\ \frac{1}{2\Dt} \int_{t^m}^{t^{m+1}} (t^{m+1}-\tau)(t^m-\tau)\dddot{u}(\tau) \d{\tau}.
	\end{split}
	\end{equation}
The absolute value of $\tau\mapsto (t^{m+1}-\tau)(t^m-\tau)$ with $\tau\in \left[t^m, t^{m+1}\right]$ is bounded,
	\begin{equation}
	\label{eq:r1mineq5}
		\left|(t^{m+1}-\tau)(t^m-\tau)\right| \leq \frac{1}{4}\Dt ^2,\qquad\tau\in [t^m, t^{m+1}].
	\end{equation}
We take the norm $\norm{\cdot}_{{V^\varrho_h}^\ast}$ of the result of \eqref{eq:r1mineq4}, use the Bochner integrability of $\dddot{u}$ guaranteed by Proposition 1.2.3 in \cite{ArendtBattyHieberNeubrander2011}, apply again 
the results from \cite[Section 24.2]{Wloka-english} and then get by invoking estimate \eqref{eq:r1mineq5} that
	\begin{align}
		\norm{\frac{1}{\Dt} \left(u^{m+1} - u^m\right) - \dot{u}^{m+\theta}}_{{V^\varrho_h}^\ast}  =&\ \frac{1}{2\Dt} \norm{\int_{t^m}^{t^{m+1}} \left(t^{m+1} - \tau\right)\left(t^m - \tau\right)\dddot{u}(\tau) \d{\tau}}_{{V^\varrho_h}^\ast}\notag\\
		\leq&\ \frac{1}{2\Dt} \int_{t^m}^{t^{m+1}} \frac{1}{4}\Dt^2 \norm{\dddot{u}(\tau)}_{{V^\varrho_h}^\ast}  \d{\tau}\notag \\
		\leq &\ \frac{1}{8}\Dt^\frac{3}{2} \left(\int_{t^m}^{t^{m+1}} \norm{\dddot{u}(\tau)}^2_{{V^\varrho_h}^\ast}  \d{\tau} \right)^\frac{1}{2},\label{eq:r1mineq6}
	\end{align}
with the H\"older inequality yielding the last step. Setting
	\begin{equation*}
		C_{r_1} = \max\left\{C_\theta,\frac{1}{8}\right\} = C_\theta\in \left[\frac{1}{2},1\right]
	\end{equation*}
defines the constant in \eqref{eq:UpperBoundsforresidualsr1} and completes the estimation of $r_1^m$.	

Second, we derive an upper bound for $\norm{r_2^m}_{{V^\varrho_h}^\ast}$.
With assumption~\ref{enum:uC2} we have again by combining absolute continuity of continuously differentiable functions with Proposition~1.2.3 in \cite{ArendtBattyHieberNeubrander2011} that $\dot{u}$ is Bochner integrable and
	\begin{equation}
	\label{eq:r2mineq1}
		u^{m+1} - u^m = \int_{t^m}^{t^{m+1}} \dot{u}(\tau)\d{\tau}.
	\end{equation}
We begin estimating the norm of $r_2^m$ by using \eqref{eq:r2mineq1} and obtain
	\begin{equation}
	\label{eq:r2mineq2}
		\begin{split}
			|(r_2^m, v_h)| \leq &\ \frac{1}{\Dt} \norm{\left(u^{m+1} - u^m\right) - P_h\left(u^{m+1} - u^m\right)}_{{V^\varrho_h}^\ast}  \norm{v_h}_{V^\varrho}\\
				=&\ \frac{1}{\Dt} \norm{(I-P_h)\left(\int_{t^m}^{t^{m+1}} \dot{u}(\tau)\d{\tau}\right)}_{{V^\varrho_h}^\ast}  \norm{v_h}_{V^\varrho},
		\end{split}
	\end{equation}
where $I$ denotes the identity mapping. By Proposition~1.1.6 in \cite{ArendtBattyHieberNeubrander2011} we may interchange integration with the $(I-P_h)$ operator to get
	\begin{equation}
	\label{eq:r2mineq3}
		(I-P_h)\left(\int_{t^m}^{t^{m+1}} \dot{u}(\tau)\d{\tau}\right) = \int_{t^m}^{t^{m+1}}(I-P_h)\left(\dot{u}(\tau)\right)\d{\tau}.
	\end{equation}
Proposition~1.1.6 in \cite{ArendtBattyHieberNeubrander2011} also grants that with $\dot{u}$ being Bochner integrable, $\left(I-P_h\right)(\dot{u})$ is Bochner integrable, as well. Consequently, we may combine \eqref{eq:r2mineq2} and \eqref{eq:r2mineq3} and conclude again by 
the results from \cite[Section 24.2]{Wloka-english} that 
	\begin{equation}
	\label{eq:r2mineq4}
		\begin{split}
			|(r_2^m, v_h)| \leq &\ \frac{1}{\Dt} \norm{(I-P_h)\left(\int_{t^m}^{t^{m+1}} \dot{u}(\tau)\d{\tau}\right)}_{{V^\varrho_h}^\ast}  \norm{v_h}_{V^\varrho}\\
			=&\ \frac{1}{\Dt} \norm{\int_{t^m}^{t^{m+1}}(I-P_h)\left(\dot{u}(\tau)\right)\d{\tau}}_{{V^\varrho_h}^\ast}  \norm{v_h}_{V^\varrho}\\
			\leq&\ \frac{1}{\Dt} \int_{t^m}^{t^{m+1}} \norm{(I-P_h) \dot{u}(\tau)}_{{V^\varrho_h}^\ast} \d{\tau} \norm{v_h}_{V^\varrho}.
		\end{split}
	\end{equation}
At this point we want to apply the approximation property of the projector $P_h$ outlined in Assumption~\ref{ass:GeneralApproximationProperty}. Before we can do that we need to establish a relation between $\norm{\cdot}_{{V^\varrho_h}^\ast}$ $\norm{\cdot}_{{V^\varrho_h}^\ast}$ and $\norm{\cdot}_{V^\varrho}$. To do so, keep $\tau\in \left[t^m,t^{m+1}\right]$ arbitrary but fix. Using the definition of norm $\norm{\cdot}_{{V^\varrho_h}^\ast}$ we derive
	\begin{equation}
	\label{eq:r2mineq5}
	\begin{split}
		 \norm{(I-P_h) \dot{u}(\tau)}_{{V^\varrho_h}^\ast} =&\ \sup\limits_{v_h\in V_h} \frac{\left((I-P_h)\dot{u}(\tau),v_h\right)}{\norm{v_h}_{V^\varrho}} \\
		 \leq&\ \sup\limits_{v_h\in V_h} \frac{\norm{(I-P_h)\dot{u}(\tau)}_H\norm{v_h}_H}{\norm{v_h}_{V^\varrho}} \\
 		 =&\ \norm{(I-P_h)\dot{u}(\tau)}_H \sup\limits_{v_h\in V_h} \frac{\norm{v_h}_H}{\norm{v_h}_{V^\varrho}} \\
 		 \leq&\ \norm{(I-P_h)\dot{u}(\tau)}_{V^\varrho},
	\end{split}
	\end{equation}
since $\norm{v}_H\leq\norm{v}_{V^\varrho}$ for all $v\in V^\varrho$.
Inserting \eqref{eq:r2mineq5} into \eqref{eq:r2mineq4} and applying the approximation property of $P_h$ pointwise in time we derive
	\begin{equation}
	\begin{split}
		|(r_2^m, v_h)| \leq &\ \frac{1}{\Dt} \int_{t^m}^{t^{m+1}} \norm{(I-P_h) \dot{u}(\tau)}_{{V^\varrho_h}^\ast} \d{\tau} \norm{v_h}_{V^\varrho}\\
		\leq &\ \frac{1}{\Dt} \int_{t^m}^{t^{m+1}} \norm{(I-P_h)\dot{u}(\tau)}_{V^\varrho} \d{\tau} \norm{v_h}_{V^\varrho}\\
		\leq &\ C_\Upsilon\,\frac{1}{\Dt} \int_{t^m}^{t^{m+1}} \Upsilon\left(h, t, \orderAhalf, \dot{u}(\tau)\right)\d{\tau}\,\norm{v_h}_{V^\varrho}\\
		\leq &\ C_{r_2}\,\frac{1}{\sqrt{\Dt}} \left(\int_{t^m}^{t^{m+1}} \Upsilon^2\left(h, \orderAhalf, t, \dot{u}(\tau)\right)\d{\tau}\right)^\frac{1}{2}\norm{v_h}_{V^\varrho},
	\end{split}
	\end{equation}
where the H\"older inequality grants the last step and where we used the additional smoothness in the sense of assumption~\ref{enum:uinWVt} and where $C_{r_2}=C_\Upsilon$.

Third, we derive and upper bound for $\norm{r_3^m}_{{V^\varrho_h}^\ast}$.
The bound for the norm of $r_3^m$ is a direct consequence of the uniform continuity of $a_t(\cdot, \cdot)$. We compute for $v_h\in V^\varrho_h$ that
\begin{align*}
	\left|\left(r_3^m,v_h\right)\right| =&\ \left|a^{m+\theta}(P_h u^{m+\theta} - u^{m+\theta},v_h)\right| \\
	\leq&\ \alpha \norm{P_h u^{m+\theta} - u^{m+\theta}}_{V^\varrho}\norm{v_h}_{V^\varrho}\\
	\leq&\ C_{r_3}\,\Upsilon(h, \orderAhalf, t, u^{m+\theta})\norm{v_h}_{V^\varrho},
\end{align*}
wherein $C_{r_3} = \alpha C_\Upsilon$, with $\alpha$ the continuity constant of $a_t(\cdot,\cdot)$ and $C_\Upsilon$ the constant stemming from the approximation property \eqref{eq:GeneralApproximationProperty} of Assumption \ref{ass:GeneralApproximationProperty}.
This finishes the derivation of upper bounds for the norms of the individual residuals $r_1^m,\ r_2^m$ and $r_3^m$, $m=0,\dots,M-1$.
\end{proof}

\subsection{Proof of the Main Convergence Result}\label{sec:proofthrm:ConvergenceCoercive}
\begin{proof}[of Theorem \ref{thrm:ConvergenceCoercive}]
For $m\in\{0,\dots,M\}$ recall the definition
	\begin{equation*}
		e_h^m = u^m - u_h^m= \eta^m + \xi_h^m
	\end{equation*}
with 
	\begin{align}
		\eta^m =&\ u^m - P_h u^m,\qquad \forall m\in\{0,\dots, M\}, \label{eq:defetamRep}\\
		\xi_h^m =&\ P_h u^m - u_h^m,\qquad \forall m\in\{0,\dots, M\},\label{eq:defximRep}
	\end{align}
as introduced in \eqref{eq:defehm}. Additionally, we denote
	\begin{align}
		\eta^{m+\theta} =&\ \theta\eta^{m+1}+(1-\theta)u^m = u^{m+\theta} - P_h u^{m+\theta},\qquad \forall m\in\{0,\dots, M-1\}, \label{eq:defetamtheta}\\
		\xi_h^{m+\theta} =&\ \theta\xi_h^{m+1} + (1-\theta)\xi_h^m = P_h u^{m+\theta} - u_h^{m+\theta},\qquad \forall m\in\{0,\dots, M-1\}. \label{eq:defximtheta}
	\end{align}

By the third binomial formula we get	
	\begin{align}
		\norm{u^M - u_h^M}_H^2 +&\ \Dt\sum_{m=0}^{M-1} \norm{u^{m+\theta} - u_h^{m+\theta}}_{a^{m+\theta}}^2\notag\\
		=&\ \norm{e_h^M}_H^2 + \Dt \sum_{m=0}^{M-1}  \norm{e_h^{m+\theta}}_{a^{m+\theta}}^2\notag\\
		=&\ \norm{\eta^M + \xi_h^M}_H^2 + \Dt \sum_{m=0}^{M-1}  \norm{\eta^{m+\theta} + \xi_h^{m+\theta}}_{a^{m+\theta}}^2\notag\\
		\leq&\ 2\left(\norm{u^M - P_h u^M}_H^2 + \Dt \sum_{m=0}^{M-1} \norm{u^{m+\theta} - P_h u^{m+\theta}}_{a^{m+\theta}}^2 \right)\label{eq:thrmconvcoerc1.1}\\
		&\ \qquad + 2\left(\norm{\xi_h^M}_H^2 + \Dt \sum_{m=0}^{M-1} \norm{\xi_h^{m+\theta}}_{a^{m+\theta}}^2\right).\label{eq:thrmconvcoerc1.2}
	\end{align}
Considering the first main summand, that is \eqref{eq:thrmconvcoerc1.1}, we simply exploit the continuity of $a_t(\cdot,\cdot)$ to get
	\begin{equation}
	\label{eq:thrmconvcoerc2}
	\begin{split}
		\norm{u^M - P_h u^M}_H^2 +&\ \Dt \sum_{m=0}^{M-1} \norm{u^{m+\theta} - P_h u^{m+\theta}}_{a^{m+\theta}}^2\\
		\leq&\ \norm{u^M - P_h u^M}_{V^\varrho}^2 + \alpha\frac{T}{M} \sum_{m=0}^{M-1} \norm{u^{m+\theta} - P_h u^{m+\theta}}_{V^\varrho}^2.
	\end{split}
	\end{equation}
Considering the term $\sum_{m=0}^{M-1} \norm{u^{m+\theta} - P_h u^{m+\theta}}_{V^\varrho}^2$ in \eqref{eq:thrmconvcoerc2} we see by the linearity of the projector $P_h$ and elementary calculations that
	\begin{align*}
		\sum_{m=0}^{M-1}\norm{u^{m+\theta} - P_h u^{m+\theta}}_{V^\varrho}^2
		=&\ \sum_{m=0}^{M-1} \norm{\theta \left(u^{m+1} - P_h u^{m+1}\right) + (1-\theta) \left(u^m - P_h u^m\right)}_{V^\varrho}^2\notag\\
		\leq&\ 2\sum_{m=0}^{M-1}\left(\theta^2 \norm{u^{m+1} - P_h u^{m+1}}_{V^\varrho}^2 + (1-\theta)^2 \norm{u^m - P_h u^m}_{V^\varrho}^2\right).
	\end{align*}
We split the sum in the right hand side and replace the individual summands by the maximum summand yielding the estimate
	\begin{equation}
	\label{eq:thrmconvcoerc3.1}
	\begin{split}
		\sum_{m=0}^{M-1}\norm{u^{m+\theta} - P_h u^{m+\theta}}_{V^\varrho}^2
		\leq&\ 2\bigg(M\theta^2\max\limits_{0\leq \tau\leq T}\left(\norm{u(\tau) - P_h u(\tau)}_{V^\varrho}^2\right)\\
		&\qquad\quad + M(1-\theta)^2\max\limits_{0\leq \tau\leq T}\left(\norm{u(\tau) - P_h u(\tau)}_{V^\varrho}^2\right)\bigg)\\
		=&\ 2M\left(\theta^2 + (1-\theta)^2\right)\max\limits_{0\leq \tau\leq T} \left(\norm{u(\tau) - P_h u(\tau)}_{V^\varrho}^2\right)\\
		\leq&\ M\max\limits_{0\leq \tau\leq T} \left(\norm{u(\tau) - P_h u(\tau)}_{V^\varrho}^2\right).	
	\end{split}
	\end{equation}
Inserting \eqref{eq:thrmconvcoerc3.1} into \eqref{eq:thrmconvcoerc2} yields
	\begin{equation*}
	\begin{split}
		\norm{u^M - P_h u^M}_H^2 +&\ \Dt \sum_{m=0}^{M-1} \norm{u^{m+\theta} - P_h u^{m+\theta}}_{a^{m+\theta}}^2\\
		\leq&\ \norm{u^M - P_h u^M}_{V^\varrho}^2 + \alpha T\max\limits_{0\leq \tau\leq T} \left(\norm{u(\tau) - P_h u(\tau)}_{V^\varrho}^2\right)\\
		\leq&\ \left(1+ \alpha T\right)\max\limits_{0\leq \tau\leq T} \left(\norm{u(\tau) - P_h u(\tau)}_{V^\varrho}^2\right).
	\end{split}
	\end{equation*}
Finally, the approximation property of the projector of Assumption~\ref{ass:GeneralApproximationProperty} applied pointwise in time, and setting $\overline{C}_1 = C_\Upsilon^2(1+ \alpha T)$, yields
	\begin{equation}
	\label{eq:thrmconvcoerc5}
	\begin{split}
		\norm{u^M - P_h u^M}_H^2 + \Dt\sum_{m=0}^{M-1}\norm{u^{m+\theta} - P_h u^{m+\theta}}_{a^{m+\theta}}^2
		& \leq \overline{C}_1\,\max\limits_{0\leq\tau\leq T}\Upsilon^2(h, \orderAhalf, t, u(\tau)).
	\end{split}
	\end{equation}

Considering now the main summand in \eqref{eq:thrmconvcoerc1.2} we find applying Corollary~\ref{cor:stabxicoerc} using the positive constants $C_1$ and $C_2$ therein that
	\begin{equation}
	\label{eq:thrmconvcoerc6}
	\begin{split}
		\norm{\xi_h^M}_H^2 +&\ \Dt \sum_{m=0}^{M-1} \norm{\xi_h^{m+\theta}}_{a^{m+\theta}}^2\\
		\leq&\ \max\left\{1,\frac{1}{C_1}\right\}\left(\norm{\xi_h^M}_H^2 + \Dt\, C_1 \sum_{m=0}^{M-1} \norm{\xi_h^{m+\theta}}_{a^{m+\theta}}^2\right) \\
		\leq&\ \max\left\{1,\frac{1}{C_1}\right\}\left(\norm{\xi_h^0}_H^2 + \Dt\, C_2 \sum_{m=0}^{M-1} \norm{r_h^m}_{{V^\varrho_h}^\ast}^2\right).
	\end{split}
	\end{equation}

We investigate $\norm{\xi_h^0}_H^2$, first. By definition of $\xi_h^m$ for $m=0$ we get
	\begin{equation}
	\label{eq:thrmconvcoerc6.1}
		\begin{split}
			\norm{\xi_h^0}_H =&\ \norm{P_h u^0 - u_h^0}_H.
		\end{split}
	\end{equation}
Recall that
	\begin{equation}
	\label{eq:thrmconvcoerc6.2}
		u^0=u(t^0)=u(0)=g,
	\end{equation}
the initial condition of the original problem~\eqref{eq:evol} and further 
	\begin{equation}
	\label{eq:thrmconvcoerc6.3}
		u_h^0=g_h,
	\end{equation}
by the initial condition of the fully discrete $\theta$-Scheme~\eqref{eq:ThetaSchemeFullydiscretizedPIDE}. With inserting both \eqref{eq:thrmconvcoerc6.2} and \eqref{eq:thrmconvcoerc6.3} into \eqref{eq:thrmconvcoerc6.1} and exploiting approximation property of the projector $P_h$ of Assumption~\ref{ass:GeneralApproximationProperty} as well as the quasi-optimality of the initial condition as stated in Assumption~\ref{ass:quasioptinitial}, we find
	        \begin{equation}
			\norm{\xi_h^0}_H \leq\ \norm{P_h u^0 - g}_H + \norm{g- u_h^0}_H
			=\ \norm{u(0)- P_h u(0)}_H + \norm{g - g_h}_H		
			\end{equation}
Inserting $ \norm{g - g_h}_H\leq C_I\inf_{v_h\in V^\varrho_h}\norm{g - v_h}_H$, and $\norm{g - v_h}_H = \norm{u(0) - v_h}_H$ we obtain
			\begin{equation}
			\label{eq:thrmconvcoerc6.4}
			\begin{split}
			\norm{\xi_h^0}_H
			\leq&\ \max\limits_{0\leq\tau\leq T}\left(\norm{u(\tau) - P_h u(\tau)}_H + C_I\inf_{v_h\in V^\varrho_h}\norm{u(\tau) - v_h}_H\right)\\
			\leq&\ \max\limits_{0\leq\tau\leq T}\,C_\Upsilon(1 + C_I)\Upsilon(h, \orderAhalf, t, u(\tau))\\
			=&\ \max\limits_{0\leq\tau\leq T}\,\sqrt{\overline{C}_2}\Upsilon(h, \orderAhalf, t, u(\tau))
		\end{split}
	\end{equation}
with $\overline{C}_2 = C^2_\Upsilon(1 + C_I)^2$, having applied the approximation property~\ref{ass:GeneralApproximationProperty} of the projector $P_h$ at the end of the derivation.

Considering next the sum of normed residuals in \eqref{eq:thrmconvcoerc6} we observe that
	\begin{equation*}
	\begin{split}
		\norm{r_h^m}_{{V^\varrho_h}^\ast}^2 =&\ \norm{r_1^m + r_2^m + r_3^m}_{{V^\varrho_h}^\ast}^2
		\leq\ 4\left(\norm{r_1^m}_{{V^\varrho_h}^\ast}^2 + \norm{r_2^m}_{{V^\varrho_h}^\ast}^2 + \norm{r_3^m}_{{V^\varrho_h}^\ast}^2\right).
	\end{split}
	\end{equation*}
We insert the individual upper bounds for the normed residuals $\norm{r^m_1}_{{V^\varrho_h}^\ast}$, $\norm{r^m_2}_{{V^\varrho_h}^\ast}$ and $\norm{r^m_3}_{{V^\varrho_h}^\ast}$ that we have derived in Lemma~\ref{lem:UpperBoundsforresiduals} to find
	\begin{equation*}
	\begin{split}
		\frac{1}{4}\sum_{m=0}^{M-1} \norm{r_h^m}_{{V^\varrho_h}^\ast}^2 \leq&\ 
		 	C^2_{r_1}\sum_{m=0}^{M-1} {\begin{cases}
												 \Dt \int_{t^m}^{t^{m+1}} \norm{\ddot{u}(s)}^2_{{V^\varrho_h}^\ast} \d{s}, & \forall \theta \in [0,1]\\
												 \left(\Dt\right)^3 \int_{t^m}^{t^{m+1}} \norm{\dddot{u}(s)}^2_{{V^\varrho_h}^\ast}\d{s}, & \theta =\frac{1}{2}
				\end{cases}}\\
&\ \quad + C^2_{r_2}\sum_{m=0}^{M-1} \frac{1}{\Dt} \int_{t_m}^{t_{m+1}} \Upsilon^2(h, \orderAhalf, t, \dot{u}(\tau))\d{\tau}\\
		 &\ \quad + C^2_{r_3}\sum_{m=0}^{M-1} \Upsilon^2(h,\orderAhalf, t, u^{m+\theta})
    \end{split}
	\end{equation*}
	and thus															
    \begin{equation}
	\label{eq:thrmconvcoerc7}
	\begin{split}									
		\frac{1}{4}\sum_{m=0}^{M-1} \norm{r_h^m}_{{V^\varrho_h}^\ast}^2	
		\leq&\ 
		 	C^2_{r_1} {\begin{cases}
												 \Dt \int_0^T \norm{\ddot{u}(s)}^2_{{V^\varrho_h}^\ast} \d{s}, & \forall \theta \in [0,1]\\
												 \left(\Dt\right)^3 \int_0^T \norm{\dddot{u}(s)}^2_{{V^\varrho_h}^\ast}\d{s}, & \theta =\frac{1}{2}
												\end{cases}}\\
		 &\ \quad + C^2_{r_2} \frac{1}{\Dt} \int_0^T \Upsilon^2(h, \orderAhalf, t, \dot{u}(\tau))\d{\tau}\\
		 &\ \quad + C^2_{r_3}M\max\limits_{0\leq \tau\leq T} \Upsilon^2(h, \orderAhalf, t, u(\tau)),
	\end{split}
	\end{equation}
with positive constants $C_{r_1}, C_{r_2}, C_{r_3}$ defined in the Lemma. We return to \eqref{eq:thrmconvcoerc1.2} and invoke \eqref{eq:thrmconvcoerc5} and \eqref{eq:thrmconvcoerc6} to derive
	\begin{equation}
	\label{eq:thrmconvcoerc8}
	\begin{split}
		 \norm{u^M - u_h^M}_H^2 +&\ \Dt\sum_{m=0}^{M-1} \norm{u^{m+\theta} - u_h^{m+\theta}}_{a^{m+\theta}}^2\\
		\leq&\ 2\left(\norm{u^M - P_h u^M}_H^2 + \Dt \sum_{m=0}^{M-1} \norm{u^{m+\theta} - P_h u^{m+\theta}}_{a^{m+\theta}}^2 \right)\\
		&\ \qquad + 2\left(\norm{\xi_h^M}_H^2 + \Dt \sum_{m=0}^{M-1} \norm{\xi_h^{m+\theta}}_{a^{m+\theta}}^2\right)\\
		\leq&\ 2\overline{C}_1\,\max\limits_{0\leq\tau\leq T}\Upsilon^2(h, \orderAhalf, t, u(\tau))\\
		&\ + 2\max\left\{1,\frac{1}{C_1}\right\}\left(\norm{\xi_h^0}_H^2 + \Dt\, C_2 \sum_{m=0}^{M-1} \norm{r_h^m}_{{V^\varrho_h}^\ast}^2\right).
	\end{split}
	\end{equation}
Invoking our considerations for $\xi_h^0$ and the sum of normed residuals $r_h^m$ in \eqref{eq:thrmconvcoerc6.4} and \eqref{eq:thrmconvcoerc7} to deduce
	\begin{equation}
	\label{eq:thrmconvcoerc8.1}
	\begin{split}
		&\norm{u^M - u_h^M}_H^2 + \Dt\sum_{m=0}^{M-1} \norm{u^{m+\theta} - u_h^{m+\theta}}_{a^{m+\theta}}^2\\	
		\leq&\ 2\overline{C}_1\,\max\limits_{0\leq \tau\leq T}\Upsilon^2(h, \orderAhalf, t, u(\tau))\\
		&\ + 2\max\left\{1,\frac{1}{C_1}\right\}\Bigg(\overline{C}_2\max\limits_{0\leq\tau\leq T}\,\Upsilon^2(h, \orderAhalf, t, u(\tau))\\
		&\ \qquad\qquad\qquad\qquad + 4C_2\Bigg(C^2_{r_1} {\begin{cases}
												 (\Dt)^2 \int_0^T \norm{\ddot{u}(\tau)}^2_{{V^\varrho_h}^\ast} \d{\tau}, & \forall \theta \in [0,1]\\
												 \left(\Dt\right)^4 \int_0^T \norm{\dddot{u}(\tau)}^2_{{V^\varrho_h}^\ast}\d{\tau}, & \theta =\frac{1}{2}
												\end{cases}}\\
		 &\ \qquad\qquad\qquad\qquad\qquad\qquad + C^2_{r_2} \int_0^T \Upsilon^2(h, \orderAhalf, t, \dot{u}(\tau))\d{\tau}\\
		 &\ \qquad\qquad\qquad\qquad\qquad\qquad + C^2_{r_3}T\max\limits_{0\leq \tau\leq T} \Upsilon^2(h, \orderAhalf, t, u(\tau))\Bigg)\Bigg).
	\end{split}
	\end{equation}

For a notationally more satisfying result we define the constant
	\begin{equation}
	\label{eq:thrmconvcoerc9}
		\overline{C} = 2\max\left\{3\overline{C}_1,\ \max\left\{1,\frac{1}{C_1}\right\}\max\left\{3\overline{C}_2,\ 4C_2\max\left\{C_{r_1}^2,\ C_{r_2}^2,\ 3C_{r_3}^2T\right\}\right\}\right\}.
	\end{equation}
Clearly, $\norm{g}_{V^\varrho}=\norm{u(0)}_{V^\varrho}\leq\max\limits_{0\leq\tau\leq T}\norm{u(\tau)}_{V^\varrho}$. Thus, using \eqref{eq:thrmconvcoerc9} in \eqref{eq:thrmconvcoerc8.1} we get the estimate
\begin{equation}
	\label{eq:thrmconvcoerc10}
	\begin{split}
		\norm{u^M - u_h^M}^2 +\ \Dt\sum_{m=0}^{M-1}&\norm{u^{m+\theta} - u_h^{m+\theta}}_{a^{m+\theta}}^2\\
		\leq&\ \overline{C}\,\max\limits_{0\leq \tau\leq T}\Upsilon^2(h, \orderAhalf, t, u(\tau))\\
		+&\ \overline{C}\,{\begin{cases}
							(\Dt)^2 \int_0^T \norm{\ddot{u}(\tau)}^2_{{V^\varrho_h}^\ast} \d{\tau}, & \forall \theta \in [0,1]\\
							(\Dt)^4 \int_0^T \norm{\dddot{u}(\tau)}^2_{{V^\varrho_h}^\ast}\d{\tau}, & \theta =\frac{1}{2}\text{ and with \ref{enum:CoercCondConvuC3}}
						\end{cases}}\\
		 +&\ \overline{C}\,\int_0^T \Upsilon^2(h, \orderAhalf, t, \dot{u}(\tau))\d{\tau}
	\end{split}
	\end{equation}
which finishes the proof.
\end{proof}

\newpage
\bibliographystyle{elsarticle-harv}
\bibliography{LiteraturFKac}

\end{document}

%% file: titlepage.tex
\thispagestyle{empty}
	\begin{center}
	{\bfseries\Large Stability and convergence of Galerkin schemes for parabolic equations with application to
	Kolmogorov pricing equations in time-inhomogeneous L\'evy models\\\vspace{.1cm} (as of \today)}
		\par\bigskip
		\vspace{0.1cm}
	{\Large Maximilian Ga\ss\, 
	and  Kathrin Glau\footnote{The authors thank Oleg Reichmann and Linus Wunderlich for fruitful discussions.}}\\
		{email: k.glau@qmul.ac.uk,}
	\end{center}
\vspace{0.1cm}

	\begin{abstract}
Two essential quantities for the analysis of approximation schemes of evolution equations are \textit{stability} and \textit{convergence}. We derive stability and convergence of fully discrete approximation schemes of solutions to linear parabolic evolution equations governed by \textit{time dependent coercive operators}. We consider abstract Galerkin approximations in space combined with theta-schemes in time. The level of generality of our analysis comprises both a large class of time-dependent operators and a large choice of approximating Galerkin spaces. In particular the results apply to partial integro differential equations for option pricing in time-inhomogeneous L\'evy models and allows for a large variety of option types and models. The derivation builds on the strong foundation laid out by \cite{PetersdorffSchwab2003} who provide the respective results for the time-homogeneous case. We discuss the assumptions in the context of option pricing.
	\end{abstract}
\minisec{Keywords}
	PIDE methods, convergence analysis, time dependent integrodifferential operator, finite elements, finance
	stability estimates, convergence analysis, Galerkin scheme, finite elements, time dependent operator
	option pricing, 
time-inhomogeneous L\'evy processes, additive processes, Sato processes	
	
\minisec{MSC2010 subject classification} 
	65M12, 
 	65M60  
	91G80, 
	60G51,
	35S10, 
		  91B25  